\documentclass[fleqn]{amsart}
\usepackage[utf8]{inputenc}

\usepackage{mathtools}
\usepackage{amsfonts}
\usepackage{amsthm}

\usepackage{amssymb}
\usepackage{art_equ}
\usepackage{xcolor}

\usepackage[section]{placeins} 
\usepackage{graphicx}
\usepackage{microtype} 
\usepackage{ctable}    

\usepackage{caption}
\usepackage{cite}

\usepackage[ocgcolorlinks]{hyperref}
\hypersetup{plainpages=false, linktocpage=true, colorlinks=false,
            breaklinks=true, linkcolor=black, menucolor=black,
			urlcolor=black, citecolor=black}

\hypersetup{pdfauthor={Andreas Wachtel}}
\hypersetup{pdftitle={The Crouzeix-Raviart element for non-conforming dual mixed methods: A priori analysis}}

\providecommand{\bm}[1]{\boldsymbol{#1}} 

\newtheorem{lemma}{Lemma}[section]
\newtheorem{theorem}[lemma]{Theorem}

\newtheorem{remark}[lemma]{Remark}

\newtheorem*{remark*}{Remark}

\newcommand\cT{\mathcal{T}}
\newcommand\cE{\mathcal{E}}

\newcommand{\bnu}{\bm\nu}

\def\bV{\boldsymbol{\mathcal V}}

\def\bfg{\boldsymbol{g}}
\def\bfv{\boldsymbol{v}}

\def\div{\mathop{\mathrm{div}}\nolimits}
\def\bdiv{\mathop{\mathbf{div}}\nolimits}
\def\tr{\mathop{\mathrm{tr}}\nolimits}

\newcommand{\jump}[1]{[\![ #1 ]\!]}

\newcommand\hdiv{{H(\div; \;\Omega)}}

\newcommand\qin{\quad\text{in}\quad}
\newcommand\qon{\quad\text{on}\quad}
\newcommand\qand{\quad\text{ and }\quad}

\providecommand{\CR}{\mathrm{CR}}
\newcommand{\CRh}{\CR_{h}}

\newcommand{\bzero}{\boldsymbol0}
\newcommand{\bsi}{\bm\sigma}
\newcommand{\bta}{\bm\tau}

\newcommand{\PP}{\mathbb{P}}
\providecommand{\bSigma}{\boldsymbol{\Sigma}}

\providecommand{\abs}[1]{\left|#1\right|}
\providecommand{\set}[1]{\left\{#1\right\}}
\providecommand{\norm}[1]{\left\|#1\right\|}
\providecommand{\normNC}[1]{\norm{#1}_{nc}}
\providecommand{\signorm}[1]{\norm{#1}_{\bSigma}}
\providecommand{\Lnorm}[2]{\norm{#1}_{0,#2}}
\providecommand{\bLnorm}[2]{\big\|{#1}\big\|_{0,#2}}
\providecommand{\Hnorm}[2]{\norm{#1}_{1,#2}}
\providecommand{\Hseminorm}[2]{\left|#1\right|_{1,#2}}
\providecommand{\seminorm}[2]{\left|#1\right|_{#2}}
\providecommand{\avg}[1]{\set{#1}}
\providecommand{\PiCR}{\boldsymbol{\Pi}_{cr}}

\def\bSg{\boldsymbol{{\underline\Sigma}}}
\providecommand{\bsignorm}[1]{\norm{#1}_{\bSg}}

\providecommand{\mkred}[1]{{\color{red}{#1}}}

\newcommand{\I}{{\mathbf I}}
\newcommand{\R}{\mathbb{R}}
\newcommand{\bn}{\boldsymbol{n}}
\newcommand{\bg}{\boldsymbol{g}}
\def\bfu{\boldsymbol{u}}
\def\bff{\boldsymbol{f}}
\def\bfx{\boldsymbol{x}}

\providecommand{\aDGnorm}[1]{\seminorm{#1}{a}}
\newcommand{\onlyBs}{\mathcal{B}_s}
\newcommand{\Bs}[2]{\onlyBs\big[(#1), (#2)\big]}
\newcommand{\onlyBnc}{\mathcal{B}_{nc}}
\newcommand{\Bnc}[2]{\onlyBnc\big[(#1), (#2)\big]}

\usepackage{ifpdf} 

\title[Crouzeix-Raviart and  dual mixed methods]{The Crouzeix-Raviart element for non-conforming dual mixed methods: A priori analysis}

\author{Tom\'as P. Barrios}
\address{Departamento de Matem\'atica y F{\'\i}sica Aplicadas, Universidad Cat\'olica de la Sant\'{\i}sima Concepci\'on, Casilla 297, Concepci\'on, Chile}
\address{Grupo de Investigación en Análisis Numérico y Cálculo Científico, GIANUC2, Concepción, Chile.}
\email{tomas@ucsc.cl}

\author{J. Manuel Casc\'on}
\address{Departamento de Econom{\'\i}a e Historia Econ\'omica, Universidad de Salamanca, Salamanca 37008, Spain}
\email{casbar@usal.es}

\author{Andreas Wachtel}
\address[corresponding author]{Departamento de Matem\'aticas, ITAM, R\'io Hondo 1, Ciudad de M\'exico 01080, Mexico}
\email{andreas.wachtel@itam.mx}

\thanks{This research was partially supported by CONICYT-Chile through the FONDECYT project with No.\ 1200051 and the Direcci\'on de Investigaci\'on of the Universidad Cat\'olica de la Sant{\'\i}sima Concepci\'on. AW gratefully acknowledges the financial support by the Asociaci\'on Mexicana de Cultura A.C.}


\begin{document}
\providecommand{\bydef}{\coloneqq}%

\date{}

\begin{abstract}
	Under some regularity assumptions, we report an \emph{a priori} error analysis of a dG scheme for the Poisson and Stokes flow problem in their dual mixed formulation.
  Both formulations satisfy a Babu\v{s}ka-Brezzi type condition within the space $H(\div)\times L^2$.
  It is well known that the lowest order Crouzeix-Raviart element paired with piecewise constants satisfies such a condition on  (broken)  $H^1\times L^2$ spaces.
  In the present article, we use this pair.
  The continuity of the normal component is weakly imposed by penalising  normal jumps of the broken $H(\div)$ component.
  For the resulting methods,  we prove well-posedness and convergence with  constants independent of data and mesh size.
  We report 
  error estimates in the methods natural norms
  and
  optimal local error estimates for the divergence error. 
  In fact, our finite element solution shares for each triangle one DOF with the CR interpolant and the divergence is locally the best-approximation for any regularity.
  Numerical experiments support the findings  and  suggest that the other errors converge optimally
  even  for
  the lowest regularity solutions and 
  a crack-problem, as long as the crack is resolved by the mesh.
\end{abstract}

\maketitle

\section{Introduction}\label{section1}
  
For a simple read, we start with the Poisson's equation.
Let $\Omega$ be a bounded and simply connected domain in $\R^2$ with polygonal boundary $\Gamma$. 
Then, given $f\in L^2(\Omega)$ and $g\in H^{1/2}(\Gamma)$ we look for $u\in H^1(\Omega),$ such that
\begin{equation}\label{model}
  -\Delta u = f \quad\text{in}\quad \Omega
  \qquad\text{and}\qquad u =
  g \quad\text{on}\quad \Gamma \,. 
\end{equation}

We follow \cite{SIAM-a-2001} and introduce the gradient $\bsi \bydef -\nabla u$ in $\Omega$ as an additional unknown. 
In this way, \eqref{model} can be reformulated as the following
problem in $\bar{\Omega}$:
\textit{Find $(\bsi,u)$ in appropriate spaces, such that}
\begin{equation}\label{dualmodel}
\bsi+\nabla u=0 \qin \Omega\;,\quad \div\bsi=f\qin \Omega\qand
 u=g\qon\Gamma\,.
\end{equation}
In the classical approach \cite{SIAM-a-2001} the scalar-valued unknown $u$ is sought in broken $H^1$ and the vectorial unknown $\bsi$ in $L^2.$
In the present work, we are interested in approximating $\bsi$  in a discrete space  that locally belongs to $H(\div)$ and the scalar-valued $u\in L^2.$
This allows to use piecewise constants to approximate $u$ without slowing down convergence for $\bsi,$ as it would in \cite{SIAM-a-2001}.

This kind of approach has been also applied in the previous works \cite{bb-2006,  bb-2010, bb-2012} and  \cite{bbs-2017}. 
All of them consider the standard local Raviart-Thomas spaces, to approximate $\bsi.$

Alternatively, it is possible to use other pairs which weakly  preserve the continuity of normal components of the vectorial unknown.
This  motivates us to approximate $\bsi$ using the space that locally is the lowest order Crouzeix-Raviart (CR) element.
This non-conforming element was introduced in the early work  \cite{Crouzeix73} in the framework of the Stokes problem analysed in the classical velocity pressure formulation.
Since then,  many advances have been made using this element, for an overview we refer to \cite{Brenner2015CR}.
Sometimes, the CR-element is considered to be a special case of DG methods (on conforming meshes),
where its use eliminates  some terms of the stabilised bilinear form.
However, the proof of the DG methods mentioned earlier (\emph{e.g.} \cite{bbs-2017, BB20}) cannot be extended easily to this element, since the Raviart-Thomas space is not a subspace of the CR-space.
Up to our best knowledge, there is no analysis developed considering this element for the
Stokes problem formulated with the unknowns velocity and pseudostress, \emph{i.e.}, in its dual mixed formulation. 
However, a conforming scheme for this approach was introduced in \cite{ctvw2010}.

Then, the proposal of this article is twofold: 
First, we explore the use of the CR-element in order to approximate the $H(\div)$ space for the mixed Poisson problem.
After that, we extend the approach to approximate the solution of the Stokes equations in their velocity-pseudostress formulation.
In both cases, we do require a conforming mesh
and the discrete approximation of the divergence is the local best-approximation for any regularity.
In numerical experiments we confirm that the approximation quality of the Raviart-Thomas space (for least regularity) can be improved within the CR space when weakly imposing the continuity of the normal component of $\bsi.$

The paper is organised as follows.
We end this section introducing notation.
Sections \ref{section-2}, \ref{section-3} and \ref{secStabiblityErrors} contain the design of the method, the existence and uniqueness of the discrete solution and stability, as well as, \emph{a priori} estimates.
	 In Section 5 we extend the applicability of the Crouzeix-Raviart element to Stokes system approximated by the unusual velocity-pseudostress formulation. 
Numerical examples are reported in Section \ref{secNumerics}.

In the rest of the paper we  will use the following  notation. 
Given any Hilbert space $H,$ we denote by $H^2$ the space of vectors of length $2,$ 
and by $H^{2\times2}$ the space of tensors, all with entries in $H.$ 
Tensor notation will be used in Section~\ref{secStokes}.
We also use the standard notation for Sobolev spaces and norms. 
In particular, let 
$\hdiv \bydef\set{\bta\in[L^2(\Omega)]^{2}\colon\div(\bta)\in L^2(\Omega)}.$
Finally, we use $C$ or $c,$ with or without subscripts, to denote generic constants, independent of the discretization parameters, which may take different values at different occurrences.

%
\section{The new non-conforming formulation}\label{section-2}
In this section, we derive a discrete formulation for the linear model \eqref{model}, applying an unusual  discontinuous Galerkin method in divergence form. 
We begin with some definitions and notations.

\subsection{Meshes}
Let $\{\cT_h\}_{h>0}$ be a conforming family of triangulations of $\bar{\Omega}$  made up of straight-side triangles $T$ with diameter $h_T$ and unit outward normal to $\partial T$ denoted by $\bm{\nu}_T.$
As usual, the index $h$ also denotes $h\bydef\max_{T\in\cT_h}h_T.$
Then, given $\cT_h$, its edges are defined as follows. 
An \emph{interior edge} of $\cT_h$ is the (non-empty) interior of $\partial T\cap\partial T',$ where $T$ and $T'$ are two adjacent elements of $\cT_h.$
We denote by $\cE_I$ the set of all interior edges of $\cT_h$  in $\Omega$ and  by $\cE_\Gamma$  the list of all boundary edges, respectively. 
Then, $\cE\bydef\cE_I\cup\cE_\Gamma$ denotes the skeleton  inherited from the triangulation $\cT_h.$
Moreover, we denote by $|e|$ the length of an edge $e\in\cE$ and by $|T|$ the area of $T\in\cT_h.$

\subsection{Averages and normal jumps}
Now, in order to define average and jump operators, let $T$ and $T'$ be two adjacent elements of $\cT_h$ and $\bfx$ be an arbitrary point on the interior edge $e=\partial T\cap\partial T'\in\cE_I$. 
In addition, let $q$ and $\bfv$ be scalar- and  vector-valued functions, respectively, that are smooth inside each element $T\in\cT_h.$
We denote by $(q_{T,e}, \bfv_{T,e})$ the restriction of $(q_T, \bfv_T)$ to $e.$
Then, we define the averages at $\bfx\in e$ by:
$\{q\}\bydef\frac{1}{2}\big(q_{T,e}+q_{T',e}\big)$ and $\{\bfv\}\bydef\frac{1}{2}\big(\bfv_{T,e}+\bfv_{T',e}\big).$
Similarly, the jumps at $\bfx\in e$ are given by
$
\jump{q}\bydef q_{T,e}\,\bnu_T + q_{T',e}\,\bnu_{T'}$
and $\jump{\bfv}\bydef \bfv_{T,e}\cdot\bnu_T + \bfv_{T',e}\cdot\bnu_{T'}.$
On boundary edges $e$, we set $\{q\}\bydef q$, 
$\{\bfv\}\bydef \bfv$, as well as $\jump{q}\bydef q\,\bnu$
and $\jump{\bfv}\bydef \bfv\cdot\bnu.$
These operators use traces and normal-traces of funcions on edges which are well-defined for functions in the broken space $H^\epsilon(\cT_h)\bydef \set{v\in L^2(\Omega)\colon v\in H^\epsilon(T) \;\text{for all}\; T\in\cT_h}$ with $\epsilon>1/2.$ 
We refer to \mkred{\cite[p.6]{GR86}} for a definition of $H^\epsilon(\Omega).$
 Since our approach is non-conforming, we also introduce 
the elementwise divergence operator $\div_h$
and the broken Sobolev space $H(\div;\cT_h)\bydef \set{\bta\in L^2(\Omega)^2\colon \div\bta\in L^2(T) \;\text{for all}\; T\in\cT_h}$ defined in the standard way.
Throughout, in order to shorten notation, we use the subspace
\[
	\bSigma\bydef\textstyle \set{\bta\in H(\div;\cT_h) \colon \int_e\jump{\bta} = 0 \;\forall e\in\cE_I }
	\,.
\]

\subsection{Discrete spaces}
We define  discrete spaces as follows.
Let $\PP_\ell(T)$ the space of polynomials of degree $\ell$ on $T\in\cT_h.$
To shorten definitions  
we define the discontinuous space
	  $\PP_\ell(\cT_h) \bydef \set{ v\in L^2(\Omega) \colon \left.v\right|_T \in\PP_\ell(T) \;\forall T\in\cT_h}.$
Then, the scalar- and vector-valued CR-spaces use the CR-element in each component as follows:
\[
  \begin{aligned}
		  \CRh &\bydef \textstyle\set{ v\in\PP_1(\cT_h) \colon \int_{e}\jump{v}=0\;\forall e\in\cE_I} \,,\\
	\bSigma_h &\bydef \set{\bta\in [H(\div, \cT_h)]^2\colon
	  \bta \in\left[\CRh\right]^2}\,.\\
  \end{aligned}
\]
Out method will approximate $u$ within the space  $V_h \bydef \PP_0(\cT_h)$
	  and $\bsi$ in 
	  $\bSigma_h\subset\bSigma.$

\subsection{The weak formulation of the Poisson problem}\label{sec_weakPoisson}
The global weak formulation of \eqref{dualmodel} is obtained as usual and reads:
  \textit{Find $(\bsi,u)\in H(\div;\Omega)\times L^2(\Omega)$, such that}
\begin{subequations}
		\begin{align}
			\label{eqweakPoissonA}
		&&  \int_{\Omega} \bsi\cdot\bta - \int_{\Omega} u\,\div\bta
		  &= -\int_\Gamma g (\bta\cdot\bm{\nu})&& \forall\,\bta\in H(\div;\Omega)\,, \\
			\label{eqweakPoissonB}
		&&  - \int_{\Omega} v\, \div\bsi&= -\int_{\Omega} f v && \forall\,v\in L^2(\Omega)\,.
		\end{align}
\end{subequations}

In order to derive the discrete scheme, we consider the first equation in \eqref{dualmodel} on each $T\in\cT_h,$  multiply by a test function in $H(\div, \cT_h)$ and  integrate by parts, to deduce
\[
		\int_T \bsi\cdot\bta - \int_Tu\,\div\bta 
+\int_{\partial T}u (\bta\cdot\bm{\nu}) = 0 \qquad \forall\,\bta\in H(\div;\cT_h) \,.
\]
Summing up the boundary terms we get
\[
		\sum_{T\in\cT_h} \int_{\partial T}u (\bta\cdot\bm{\nu}) =
		\sum_{e\in\cE_I}\int_{e} \Big[ u_-(\bta_-\cdot\bm{\nu_-}) + u_+(\bta_+\cdot\bm{\nu}_+)\Big]
   + \sum_{e\subset\Gamma}\int_{e} g(\bta\cdot\bm{\nu})  \,,
\]
where the signs denote the restrictions to the two neighbouring cells in $\cT_h.$
On each interior edge, we have the standard dG identity
\[
  \int_{e} u_-(\bta_-\cdot\bm{\nu_-}) + u_+(\bta_+\cdot\bm{\nu}_+)
  = \int_{e} \jump{\bta}\avg{u} + \avg{\bta}\jump{u} 
  \,.
\]
Now, since the solution $u$ belongs to $H^{t}(\Omega), \,t>1/2,$  the jump $\jump{u}$ vanishes  on interior edges $\cE_I.$
This means, the global weak  formulation for test functions in $\bta \in H(\div,\cT_h)$ simplifies to
\begin{equation}
		  \int_{\Omega} \bsi\cdot\bta - \int_{\Omega} u\,\div_h\bta
  + \sum_{e\in\cE_I}\int_{e} \jump{\bta}\avg{u}
  = - \int_{\Gamma} g(\bta\cdot\bm{\nu})\,. 
		\label{eqWeakPoissonNC}
\end{equation}
Finally, we realise that $\int_{e}\jump{\bta_h}\avg{\tilde{u}}=0$ for all  $\bta_h\in\bSigma_h$ and any approximation $\tilde{u}\in V_h$  of $u.$
Therefore, these terms will not appear in our discrete scheme.

\subsection{The global discrete scheme}
Considering \eqref{eqweakPoissonB} and \eqref{eqWeakPoissonNC} 
we arrive at the following stabilised discrete dual mixed non-conforming Galerkin formulation:
\textit{Find $(\bsi_h,u_h)\in\bSigma_h\times V_h$, such that}
\begin{subequations}\label{LDG-form1}
\begin{align}
 \label{eqLDGform1a}
&&  a_{s}(\bsi_h,\bta) - b(\bta,u_h)&=-G(\bta)  \qquad\forall\; \bta\in  \bSigma_{h} \\[1ex]
&&  -b(\bsi_h,v) &=-F(v) \qquad\forall\; v\in V_h \label{eqL2projsigmah}
\end{align}
\end{subequations}
where  the bilinear forms $a_{s}\colon\bSigma\times \bSigma\to\R$
and
$\,b\colon \bSigma\times L^2(\Omega)\to\R\,$ are defined by
\[
  a_{s}(\bsi,\bta) \bydef  \int_{\Omega} \bsi\cdot\bta 
  +
  \sum_{e\in\cE_I}\int_e\gamma_e\jump{\bsi}\jump{\bta}
  \qand
  b(\bta,v)\bydef \int_{\Omega} v\,\div_h\bta \,,
\]
where  $\gamma_e= 1/\abs{e}$ on each edge.
The linear functionals $G\colon\bSigma\to\R$ and $F\colon L^2(\Omega)\to\R$ are given by
\[
  G(\bta)\bydef  \int_{\cE_\Gamma}g\,(\bta\cdot\bm{\nu}) \qand F(v)\bydef  \int_{\Omega} f\,v \,. 
\]

\begin{remark}\label{rem_choiceGammas}
		The jumps  of the normal components of $\bsi, \bta$ in the bilinear form $a_{s}$ will be used 
  to control the consistency error  visible when comparing  \eqref{eqWeakPoissonNC} and \eqref{eqLDGform1a}.
  The value chosen for $\gamma_e$ allows this error to decrease with the optimal speed (as in dG methods).
  Additionally, this value is later shown (Lemma~\ref{lem_infsup2}) to not affect the existence of the discrete solution in a negative fashion.
Even, without penalising the jumps,
  $\bta_h$ can be thought of as weakly belonging to $H(\div, \Omega),$ since $\int_e \jump{\bta_h}=0$ for all $e\in\cE_I.$
  However, the normal component of functions $\bsi_h\in\bSigma_h$  jumps across edges
  and the penalty reduces this behaviour and the consistency error.
\end{remark}

For arguments below we will use the following equivalent form of problem \eqref{LDG-form1}:
\textit{Find  $(\bsi_h,u_h)\in\bSigma_h\times V_h$, such that}
\begin{equation}
	\Bs{\bsi_h,u_h}{\bta,v}  = -G(\bta) - F(v)
  \quad \forall (\bta,v)\in\bSigma_h\times V_h
  \label{eqCompleteSystem}
\end{equation}
where
\begin{equation}
		\Bs{\bsi_h,u_h}{\bta,v}  
		\bydef  a_{s}(\bsi_h, \bta) - b(\bta,u_h) - b(\bsi_h,v) \,.
  \label{eq_bilinearFormBs}
\end{equation}

In order to study the existence, uniqueness and approximation qualities of the discrete solution 
we introduce the following norms and properties of the CR interpolant.
The space $\bSigma$ is equipped with the norm $\signorm{\cdot}\colon \bSigma\to \R,$ which is defined by
\[
	\signorm{\bta}^2 \bydef \aDGnorm{\bta}^2 +  \Lnorm{\div_h\bta}{\Omega}^2
	\quad\forall\,\bta\in\bSigma\,,
\]
where  $\aDGnorm{\cdot}\colon \bSigma \to \R$ is defined by
\begin{equation}  \label{defAnormLap}
  \aDGnorm{\bta}^2 \bydef a_{s}(\bta,\bta) = \Lnorm{\bta}{\Omega}^2 + \bLnorm{\gamma^{1/2}\jump{\bta}}{\cE_I}^2 
\quad\forall\,\bta\in\bSigma\,.
\end{equation}
We note that the jumps vanish for $\bta\in\bSigma\cap H(\div, \Omega).$
For $V_h$ we use the standard $L^2$-norm. 
In addition, we define the natural norm $\normNC{(\cdot,\cdot)}\colon\bSigma\times L^2(\Omega)\to\R$ of the method by
\begin{equation}
	\normNC{(\bta,v)}^2 \bydef \signorm{\bta}^2 + \Lnorm{v}{\Omega}^2
\quad\forall\,(\bta,v)\in \bSigma\times L^2(\Omega)\,.
	\label{defNCnorm}
\end{equation}

\subsection{Properties of the Crouzeix-Raviart interpolant}

 Given $\varepsilon>\frac{1}{2},$
 the CR-space and approximation operator $\Pi_{cr}\colon H^{\varepsilon}(\Omega) \to \CRh$ was first introduced in \cite{Crouzeix73}
and is defined as follows:
\begin{equation}
  \int_{e} \Pi_{cr} w = \int_{e} w \qquad \forall e\in\cE\,.
  \label{eq_defPiCR}
\end{equation}
The vector-valued operator 
$\PiCR\colon \left[H^{\varepsilon}(\Omega)\right]^2 \to \bSigma_h$ 
uses the same definition for each component.
A useful property of this operator is the following: 
\begin{equation}
  \label{eqPiCRdivergence}
\int_T\div (\PiCR\bsi) = \int_T\div \bsi 
\quad \text{for all}\quad \bsi\in H(\div,\Omega)\cap [H^{\varepsilon}(\Omega)]^2, T\in\cT_h.
\end{equation}

Furthermore, we will use the following inf-sup condition.
\begin{lemma}\label{lem_infsupCR} 
  There exists a constant $C_{\Omega}>0$ only depending on the domain, such that 
for every $v\in V_h,$
there exists $\bta \in \bSigma_h$ such that 
\[
\int_\Omega \div_h(\bta) v = \Lnorm{v}{\Omega}^2 
\qquad\text{and}\qquad
\Hnorm{\bta}{\cT_h} \leq C_{\Omega}\Lnorm{v}{\Omega}\,.
\]
\end{lemma}
\begin{proof}
		By  \cite{GR86} this result is equivalent to 
		the inf-sup condition proved in \cite{Crouzeix73}.
\end{proof}

\subsection{A modified inf-sup condition.}
For the existence and uniqueness of $(\bsi_h, u_h)$ we will show that $\bta$ in Lemma~\ref{lem_infsupCR} satisfies $\signorm{\bta}\leq C\Hnorm{\bta}{\cT_h}$ for appropriate parameters $\gamma_e.$
Lemma~\ref{lem_infsup2} proves this bound using the following lemma.

\begin{lemma}\label{lem_traceInverse}
  Let $T$ be a triangle with edge $e$ and  let $v\in\mathbb{P}_k(T)$, then
  \[
	\Lnorm{v}{e}^2 \leq \frac{1}{2}(k+1)(k+2)\frac{|e|}{|T|}\Lnorm{v}{T}^2 \,.
  \]
\end{lemma}
\begin{proof}
  The penultimate bound  in the proof of \cite[Theorem 3]{WH03} states this result.
\end{proof}

\begin{lemma}\label{lem_infsup2} 
		Choosing  a non-negative function $\left.\gamma\right|_e \leq  \frac{1}{|e|}\min\set{|T|h_T^{-2}, |T'|h_{T'}^{-2}},$
		there exists a constant ${C_{sup}>0}$ independent of  mesh and data, such that 
for every $v\in V_h$, 
there exists $\bta \in \bSigma_h$ such that 
\[
\int_\Omega \div_h(\bta) v = \Lnorm{v}{\Omega}^2 
\qquad\text{and}\qquad
\signorm{\bta} \leq C_{sup}\Lnorm{v}{\Omega}\,.
\]
\end{lemma}
\begin{proof}\providecommand{\PiRT}{\mathrm{RT}}
		We only show that $\bta$ from Lemma~\ref{lem_infsupCR} satisfies $\signorm{\bta}\leq C\Hnorm{\bta}{\Omega}.$
		The only difficult term contains the normal jumps $\jump{\bta}$ which we bound as follows.
 First, we note that the Raviart-Thomas interpolant (of lowest order) is well defined for $\bta\in \bSigma_h\subset [H^1(\cT_h)]^2,$ since the average of $\bta$ across edge $e$ is continuous, \emph{i.e.}, $\int_e\bta|_T = \int_e\bta|_{T'}.$
   We write $\PiRT\bta$ and realise $\jump{\PiRT\bta}=0$ (pointwise) across all edges $e\in\cE_I.$
   Then
   \[
  \sum_{e\in\cE_I} \int_{e} \gamma \jump{\bta}^2
  = \sum_{e\in\cE_I} \int_{e} \gamma \jump{\bta - \PiRT\bta}^2
  \leq C_1 \sum_{T\in\cT_h}\sum_{e\subset\partial T}  \gamma|_e \frac{|e|}{|T|} \Lnorm{\bta- \PiRT\bta}{T}^2
\]
where $C_1$ depends on Lemma~\ref{lem_traceInverse} and is independent of mesh and data.
Now, a simplified version of \cite[Theorem~4.1]{AD99} gives 
$\Lnorm{\bta- \PiRT\bta}{T}^2 \leq C_2 h_T^2\Hseminorm{\bta}{T}^2$
with $C_2$ independent of~$T.$
Therefore, for  $\gamma|_e$ satisfying the hypothesis
the constant $C_{sup} \leq (1+2C_1C_2)C_{\Omega}$  is independent of mesh properties
which finishes the proof.
\end {proof}

\begin{remark*}
	Lemma~\ref{lem_infsup2} is trivial  when $\gamma|_e=0,$
	however, the upper bound 
	gives the largest possible choice that ensures the inf-sup condition \eqref{eq_iscSystem} (with a constant independent of the mesh) and therefore the well-posedness of problem \eqref{LDG-form1}.
	 In order
	to guarantee that the consistency error reduces at optimal speed we need to choose $\gamma|_e = 1/|e|$ which is close to the upper bound and allows optimal convergence.
Bigger values for $\gamma$ will reduce the inf-sup constant and deteriorate the convergence.
\end{remark*}

\section{Existence and uniqueness}\label{section-3}
The well-posedness of Problem \eqref{LDG-form1} and equivalently \eqref{eqCompleteSystem} is established in the next theorem.
The proof is well-known and included for completeness.

\begin{theorem}\label{thm_iscSistema} 
There exists a constant $\alpha>0,$ such that for all $(\bsi, u)\in\bSigma_h\times V_h$ we have
\begin{equation}
	 \alpha \normNC{(\bsi,u)} \leq
  \sup_{\bzero\neq(\bta,v)\in\bSigma_h\times V_h  }
  \frac{\Bs{\bsi,u}{\bta,v}}{\normNC{(\bta,v)}} \,.
  \label{eq_iscSystem}
\end{equation}
\end{theorem}
\begin{proof}
		Consider $(\bsi, u)\in\bSigma_h\times V_h$ fixed.
		Then, since $\div_h\bsi\in V_h$
		we take the test functions $\bta \bydef \bsi$ and $v\bydef -u-\div_h\bsi,$ to get
\begin{equation}
		\Bs{\bsi,u}{\bta,v}
  = \Lnorm{\bsi}{\Omega}^2 + \Lnorm{\div_h\bsi}{\Omega}^2 + \int_{\cE_I} \gamma \jump{\bsi}^2
  = \signorm{\bsi}^2\,.
  \label{est_positive}
\end{equation}
Now to control $\Lnorm{u}{\Omega}$ we use Lemma~\ref{lem_infsup2}, that is, there exists $\tilde\bta\in\bSigma_h$ such that
$\signorm{\tilde\bta}\leq C_{sup}\Lnorm{u}{\Omega}$ and for every positive $\delta$ to be selected later, we have
\[
  -b(-\delta\tilde\bta, u) =  \delta\Lnorm{u}{\Omega}^2 \,.
\]
Therefore, using the test functions $\bta\bydef -\delta\tilde\bta$ and $v=0$ we get
\[
  \begin{aligned}
	 \Bs{\bsi,u}{-\delta\tilde\bta,0}
  &=-\delta a_{s}(\bsi, \tilde\bta) +  \delta\Lnorm{u}{\Omega}^2
  \\
  &\geq -\delta C_{sup}\aDGnorm{\bsi}\Lnorm{u}{\Omega}+ \delta\Lnorm{u}{\Omega}^2 
  \\
  &\geq -\frac{1}{2}\aDGnorm{\bsi}^2 -\frac{\delta^2C_{sup}^2}{2}\Lnorm{u}{\Omega}^2+  \delta\Lnorm{u}{\Omega}^2 
  \,,
  \end{aligned}
\]
and choosing $\delta\bydef 1/C_{sup}^2,$ we obtain
\begin{equation}
		\Bs{\bsi,u}{-\delta\tilde\bta,0}
  \geq -\frac{1}{2}\aDGnorm{\bsi}^2  + \frac{\delta}{2}\Lnorm{u}{\Omega}^2 \,.
  \label{est_positiveLu}
\end{equation}
Finally, combining the test functions to $\bta_2 \bydef \bsi -\delta\tilde\bta$, recalling $v\bydef -u-\div_h\bsi$
and adding up \eqref{est_positive} and \eqref{est_positiveLu} yields
\[
		\Bs{\bsi,u}{\bta_2,v}
  \geq \frac{1}{2}\signorm{\bsi}^2 +\frac{\delta}{2}\Lnorm{u}{\Omega}^2
  \geq C \normNC{(\bsi,u)}^2
  \,.
\]
On the other hand, using $\signorm{\tilde\bta}\leq C_{sup}\Lnorm{u}{\Omega} = \delta^{-1/2}\Lnorm{u}{\Omega}$ gives
\[
  \begin{aligned}
  \normNC{(\bta_2, v)} 
  &= \normNC{(\bsi -\delta\tilde\bta, -u-\div_h\bsi)}
  \\
  &\leq \signorm{\bsi} + \delta\signorm{\tilde\bta} + \Lnorm{u}{\Omega} +\Lnorm{\div_h\bsi}{\Omega}
  \\
  &\leq 2\signorm{\bsi} + (1+\delta^{1/2})\Lnorm{u}{\Omega} 
  \leq C \normNC{(\bsi, u)}
  \end{aligned}
\]
which completes the proof of \eqref{eq_iscSystem}.
\end{proof}

\section{Stability and \emph{a priori} estimates}\label{secStabiblityErrors}
In this section, we focus on the stability and \emph{a priori}  error analysis for the scheme \eqref{LDG-form1}.
The main advantage of our approach and using $\bSigma_h\times V_h$ is the inf-sup condition stated in Lemma~\ref{lem_infsup2} which allows a method without stabilisation terms for $u_h.$
This is similar to the conforming Raviart-Thomas pair.
We will perform the analysis for parts  of the norm $\normNC{(\cdot,\cdot)}.$
Hereafter, $(\bsi,u)$ and $(\bsi_h,u_h)$ will be the unique solutions of \eqref{dualmodel} and \eqref{LDG-form1}, respectively.

First, we discuss stability and best-approximation results without any additional regularity assumption.
That is, the exact solution $(\bsi, u)\in\hdiv\times H^1(\Omega)$ satisfies 
\eqref{eqweakPoissonB}, 
\eqref{eqWeakPoissonNC} 
and $\jump{\bsi}=0$ (a.e.).
Taking the difference of these equations and the scheme \eqref{LDG-form1} we get the following quasi-consistency identities:
\begin{subequations}
	\label{consistencyLap}
\begin{align}
		\label{consistencyLapA}
 a_{s}(\bsi_h-\bsi,\bta) - b(\bta,u_h-u)&= \int_{\cE_I}\jump{\bta}\avg{u} \quad&& \forall\; \bta\in  \bSigma_{h}\,, \\
		\label{consistencyLapB}
 b(\bsi_h-\bsi,v) &=0 && \forall\; v\in V_h\,.
\end{align}
\end{subequations}

The following lemma bounds the consistency error shown in identity~\eqref{consistencyLapA}.
\begin{lemma}
		\label{lemConsistencyError}
		The consistency error is bounded as follows:
		\[
			\abs{\int_{\cE_I}\jump{\bta}\avg{u} } \leq C_1h \Hseminorm{u}{\Omega} \aDGnorm{\bta} \quad \forall\; \bta\in  \bSigma_{h}\,. \\
		\]
\end{lemma}
\begin{proof}
Let $\Pi u\in V_h$ be the $L^2$ projection of $u\in H^1(\Omega)$.
Since $\int_e \avg{\Pi u}\jump{\bta}=0$ for all $e\in\cE_I$ and all $\bta\in\bSigma_h$, 
we get
\[
		\begin{aligned}
				\abs{\sum_{e\in\cE_i} \int_e \avg{u} \jump{\bta}}
				&= \abs{\sum_{e\in\cE_i} \int_e \avg{u-\Pi u} \jump{\bta}} \\
			&\leq 
			\left(\sum_{e\in\cE_i} \int_e \gamma^{-1}\abs{\avg{u-\Pi u}}^2\right)^{1/2}
			\left(\sum_{e\in\cE_i} \int_e \gamma \jump{\bta}^2\right)^{1/2} \,.
		\end{aligned}
\]
The right term is part of $\aDGnorm{\bta}$ and 
the left term is bounded by $C_1h \Hseminorm{u}{\Omega}$ because of 
$\gamma_e^{-1} = |e|$ 
and  a standard trace estimate.
\end{proof}

We now state stability estimates.
\begin{lemma}
	The solution  $(\bsi_h, u_h)$ of scheme~\eqref{LDG-form1} is stable  in the following sense:
	\[
		\normNC{(\bsi_h, u_h)} \leq C \big(\Lnorm{f}{\Omega}+\norm{g}_{1/2, \partial\Omega}\big) \,.
	\]
	Furthermore, on each $\,T\in\cT_h\,$ we have $\,\Lnorm{\div_h\bsi_h}{T}\leq\Lnorm{\div\bsi}{T}=\Lnorm{f}{T}$.
	\label{lemStable}
\end{lemma}
\begin{proof}
  Since $\div_h\bsi_h\in V_h$ the second local estimate follows from \eqref{consistencyLapB} 
  using $v\bydef\div_h\bsi_h.$
		For the first bound, we consider
		the definition of $\onlyBs$ \eqref{eq_bilinearFormBs} and the quasi-consistency \eqref{consistencyLap}
		to deduce
		\begin{equation}
				\Bs{\bsi_h-\bsi, u_h-u}{\bta, v} = 
				\int_{\cE_I}\avg{u}\jump{\bta}
				\quad\text{for all}\quad
				(\bta, v)\in\bSigma_h\times V_h\,.
				\label{eqBilinearConsistency}
		\end{equation}
Therefore,  from Theorem~\ref{thm_iscSistema} 
we know there exists $\alpha>0$, such that
\begin{align}
		\nonumber
 \alpha \normNC{(\bsi_h,u_h)} 
&\leq
  \sup_{\bzero\neq(\bta,v)\in\bSigma_h\times V_h  }
  \frac{\Bs{\bsi_h,u_h}{\bta,v}}{\normNC{(\bta,v)}} \\
		\nonumber
&=
  \sup_{\bzero\neq(\bta,v)\in\bSigma_h\times V_h  }
  \frac{\Bs{\bsi,u}{\bta,v} + \int_{\cE_I}\avg{u}\jump{\bta}}{\normNC{(\bta,v)}} \\
  &\leq \normNC{(\bsi,u)} + C_1 h\Hseminorm{u}{\Omega} \,.
		\label{eqStability1}
\end{align}
The last step above follows using Lemma~\ref{lemConsistencyError}
and the continuity of $\onlyBs$ with respect to the norm $\normNC{(\cdot,\cdot)},$ which is established by Cauchy's inequality.

To obtain the result, we consider 
the definition of the norm which simplifies considering that $\bsi\in\hdiv$ (\emph{i.e.} $\jump{\bsi}=0$)
and the equations in \eqref{dualmodel}, that is
\begin{equation}
	\normNC{(\bsi,u)}^2 = \Lnorm{\bsi}{\Omega}^2 + \Lnorm{\div\bsi}{\Omega}^2 +  \Lnorm{u}{\Omega}^2
	= \Lnorm{\nabla u}{\Omega}^2 + \Lnorm{f}{\Omega}^2 +  \Lnorm{u}{\Omega}^2\,.
		\label{eqStability2}
\end{equation}
Finally, considering that $u$ satisfies $C\Hnorm{u}{\Omega}\leq \Lnorm{f}{\Omega} + \norm{g}_{1/2,\partial\Omega}$ completes the proof.
\end{proof}

\subsection{Error estimates}
We continue  with \emph{a priori} estimates.
The first one is a completely local best-approximation result.
\begin{lemma}\label{lem_ORC_divPoisson}
	The discrete solution $\bsi_h\in\bSigma_h$ satisfies
\[
\Lnorm{\div(\bsi-\bsi_h)}{T} 
= \inf_{v\in V_h} \Lnorm{v - \div\bsi}{T} 
= \inf_{v\in V_h} \Lnorm{v - f}{T} \,.
\]
In particular, if $\div\bsi = f\in H^{s}(\Omega)$ for some $0<s\leq1,$ then
\[
	\Lnorm{\div(\bsi-\bsi_h)}{T} \leq C h_T^s \norm{f}_{s,T}
\]
  with a constant $C>0$ is independent of mesh and data.
  In fact, if $s=1$ then $C=1/\pi.$
\end{lemma}
  \begin{proof}
	  The consistency identity \eqref{consistencyLapB} states
		\[
				\int_T v\, \div(\bsi -\bsi_h) = 0  
				\quad \text{for all} \quad v\in V_h \text{ and } T\in\cT_h \,.
		\]
		Since $v\in V_h$ is constant in $T$  
		and $\div\bsi_h \in V_h,$
		we have $\left.(\div\bsi_h)\right|_T = \frac{1}{|T|}\int_T  \div\bsi$
		and then $\left.(\div\bsi_h)\right|_T$ is the $L^2$-projection into local constants on $T.$
		This is the best-approximation of $\div\bsi$ in $V_h$ which gives the first equality.
		The second infimum follows considering the $\div\bsi=f$ as stated in \eqref{dualmodel}.
		Finally, we refer to \cite{GR86, gs-2004,  PW60c} for the second estimate.
  \end{proof}

\begin{remark*}
		Higher order Crouzeix-Raviart pairs also satisfy Lemma~\ref{lem_ORC_divPoisson}.
\end{remark*}

\begin{theorem}\label{thmSmoothConvergence}
	Let $\bsi\in H^2(\Omega)^2$ and $u\in H^1(\Omega)$ be the solution of \eqref{dualmodel}
	and $(\bsi_h, u_h)\in\bSigma_h\times V_h$ be the solution of \eqref{LDG-form1}.
	Then, 
	\[
		\normNC{(\bsi-\bsi_h,u-u_h)} \leq C\, h \left( \norm{\bsi}_{2,\Omega} + \Hseminorm{u}{\Omega}\right)   \,,
	\]
	where $C>0$ is a constant independent of $h,  \bsi$ and $u.$
\end{theorem}
\begin{proof}
		The following arguments are common  and included for completeness.

		As usual, we split the error in a discrete error and a projection error,
		\emph{i.e.}
		\[
				\begin{aligned}
				\bsi_h-\bsi &= e_h^{\bsi} + e^{\bsi} 
				\quad\text{with}\quad
				e_h^{\bsi} \bydef  \bsi_h - \Pi_{\sigma} \bsi 
				\quad\text{and}\quad
				e^{\bsi} \bydef \Pi_{\sigma} \bsi - \bsi \,, \\
				u_h-u &= e_h^{u} + e^{u} 
				\quad\text{with}\quad
				e_h^{u} \bydef  u_h - \Pi u 
				\quad\text{and}\quad
				e^{u} \bydef \Pi u - u \,,
				\end{aligned}
		\]
		where   $\Pi$ is the $L^2$-projection into $V_h$ and $\Pi_\sigma$ will be chosen later.
		Then, the norm allows to bound the errors separately as follows:
		\[
			  \normNC{(\bsi-\bsi_h,u-u_h)} \leq 
			  \normNC{(e_h^{\bsi},e_h^u)} +  \normNC{(e^{\bsi},e^u)} \,.
		\]
		We first bound the discrete error (the left norm), later we concern ourselves with the projection error (the right norm).
		To this end, we reuse the consistency error identity
		\eqref{eqBilinearConsistency}.
Hence,  from Theorem~\ref{thm_iscSistema} 
we know there exists $\alpha>0$, such that
\begin{align}
	\nonumber
	\alpha \normNC{(e_h^{\bsi}, e_h^u)} 
&\leq
  \sup_{\bzero\neq(\bta,v)\in\bSigma_h\times V_h  }
  \frac{\Bs{\bsi_h-\Pi_\sigma\bsi,u_h-\Pi u}{\bta,v}}{\normNC{(\bta,v)}} \\
		\nonumber
&=
  \sup_{\bzero\neq(\bta,v)\in\bSigma_h\times V_h  }
  \frac{\Bs{\bsi-\Pi_{\sigma}\bsi,u-\Pi u}{\bta,v} + \int_{\cE_I}\avg{u}\jump{\bta}}{\normNC{(\bta,v)}} \\
  &\leq \normNC{(\bsi-\Pi_\sigma\bsi,u-\Pi u)} + C_1 h\Hseminorm{u}{\Omega} \,.
	\nonumber
\end{align}
Here we
bounded the consistency error by Lemma~\ref{lemConsistencyError}.

It remains to bound each part of the norm $\normNC{\cdot}$ defined in \eqref{defNCnorm}.
 We start using a Poincar\'e estimate (or the standard $L^2$ projection error, see  \cite{PW60c}), to obtain
\[
	\Lnorm{u-\Pi u}{\Omega}^2
	= \sum_{T\in\cT_h} \Lnorm{u-\Pi u}{T}^2
	\leq
	 \sum_{T\in\cT_h} \frac{h_T^2}{\pi^2}\Hseminorm{u}{T}^2
	 \leq \frac{h^2}{\pi^2} \Hseminorm{u}{\Omega}^2\,.
\]
Since we are in the smooth case with $\bsi\in H^2(\Omega)^2$ we choose $\Pi_\sigma$ to be the  nodal interpolant into the piecewise linear continuous polinomials contained in $\bSigma_h$.
Then, we conclude using a known standard estimate that
\[
		\Lnorm{\div(\bsi-\Pi_\sigma\bsi)}{T} 
		\leq C \Hseminorm{\bsi-\Pi_\sigma\bsi}{T}
	\leq C h_T \norm{\bsi}_{2,T}
\]
Furthermore, since $\jump{\bsi}=0$ and $\jump{\Pi_\sigma\bsi}=0$ pointwise on all edges in $\cE_I$ 
we conclude
	$\bLnorm{\gamma^{1/2}\jump{\Pi_\sigma \bsi}}{\cE_I} =0.$

Joining these estimates  completes the proof.
\end{proof}

\begin{remark}
It is known that the treatment of the Darcy flow is similar to Poisson's equation formulated in mixed form.
In fact,  the scheme presented here coincides with  the scheme for the Darcy flow presented in \cite{BH2005}.
Hence, the analysis exhibited here can be extended, in a natural way, to the porous media equations.
Obviously,  adapting the analysis to the boundary conditions considered in the model, which should be clear at least for Dirichlet, Neumann and Mixed type boundary conditions.
In other words, the lowest order of  Crouzeix-Raviart element for the velocity and piecewise constants for the pressure is an inf-sup stable but inconsistent pair for Darcy's law, too.
\end{remark}

\clearpage
\section{Stokes}
\label{secStokes}

In this section we extend the method to the incompressible Stokes flow 
for a velocity $\bfu\colon\R^2\to\R^2$ and a pressure $p\colon\R^2\to\R$ with $p\in L^2_0(\Omega),$ 
\emph{i.e.}
\[
	\bdiv\big( \nu\nabla \bfu  - p \I) = -\bff
	\;,\quad
	\div\bfu = 0 
		\quad\text{in}\quad \Omega
		\qquad\text{and}\qquad
		\bfu = \boldsymbol{g} 
		\quad\text{on}\quad \Gamma
\]
with  given $\nu>0$ and source terms $\bff\in [L^2(\Omega)]^2$ and  $\bfg\in[H^{1/2}(\Gamma)]^2.$ 
Here and later on, $\I$ denotes the identity matrix and $\bdiv$ gives a vector whose 
entries are the divergence of a row of the tensor.
Furthermore, we use the following additional notation:
the space
 $H(\bdiv,\Omega)\bydef\set{\bta\in[L^2(\Omega)]^{2\times2}\colon\bdiv(\bta)\in L^2(\Omega)^2},$
 the broken space $H(\bdiv,\cT_h),$ 
for a tensor $\bta$ we set $\bta_{T,e}$ the restriction $\bta_T$ to $e,$
and define the jump and average  
by
$
\jump{\bta}\bydef \bta_{T,e}\,\bnu_T + \bta_{T',e}\,\bnu_{T'}$
and
$\{\bta\}\bydef\frac12\big(\bta_{T,e}+\bta_{T',e}\big),$
on $e\in\cE_I$
and by 
$\jump{\bta}\bydef \bta\,\bnu$ and  $\{\bta\}\bydef \bta$
on the boundary.
Finally, 
let $\bdiv_h$ denote the element-wise divergence operator and
$\underline{\bSigma}\bydef\textstyle \set{\bta\in H(\bdiv;\cT_h) \colon \int_e\jump{\bta} = \boldsymbol0\;\forall e\in\cE_I }.$

Now, in order to use the dual mixed 
formulation
we introduce the pseudostress
$\bsi \bydef \nu\nabla \bfu  - p \I$ in $\Omega.$ 
Taking the trace and applying $\div\bfu=0$ we realise that $p = -\frac12\tr{\bsi}$ (independent of variations of $\nu$).
This relation and $p\in L^2_0(\Omega)$
give
$\bsi\in H_0$ where $H_0\bydef \set{\bta\in H(\bdiv,\Omega)\colon \int_\Omega\tr(\bta)=0} \subset\bSg_0 \bydef
\set{\bta\in\bSg\colon \int_\Omega\tr(\bta)=0}.$
Additionally, 
defining the strong  \emph{deviator} $\bta^d \bydef \bta - \frac12(\tr{\bta})\I\in H_0$
and using the previous identities we obtain
$\bsi^d = \bsi + p\I = \nu\nabla\bfu.$
Thus, we get the equivalent \emph{velocity-pseudostress} formulation:
\textit{Find} $(\bsi, \bfu)\in H_0\times\big[H^1(\Omega)\big]^2$ \textit{such that}
\begin{equation}\label{eqStokes2}
		\frac1{\nu} \bsi^d = \nabla\bfu 
		\;,\quad
		\bdiv{\bsi}=-\bff \qin\Omega
		\qquad\text{and}\qquad
		\bfu = \boldsymbol{g} 
		\qon \Gamma\,.
\end{equation}
The pressure can be reconstructed from $\tr\bsi.$
A version of the previous discussion, as well as, the weak formulation of \eqref{eqStokes2}
can be found in \cite{bbs-2017}.
The formulation reads:
\textit{Find} $(\bsi, \bfu)\in H_0\times\big[L^2(\Omega)\big]^2$ \textit{such that}
\begin{subequations}\label{stokes-eq2-weak}
		\begin{align}
				\label{stokes-eq2-weak1}
			a(\bsi, \bta) + b(\bta, \bfu) &= G(\bta) && \forall \bta\in H_0\,,\\
				\label{stokes-eq2-weak2}
			b(\bsi, \bfv) &= F(\bfv) && \forall \bfv\in [L^2(\Omega)]^2\,.
		\end{align}
\end{subequations}
The bilinear forms $a\colon \bSg\times \bSg\to\R$ and $b\colon \bSg\times L^2(\Omega)^2\to\R$ are defined by
\[
	a(\bsi, \bta) \bydef \int_\Omega \frac1{\nu} \bsi^d : \bta^d
	\qand
	b(\bta, \bfv) \bydef \int_\Omega \bfv \cdot \bdiv(\bta)
\]
where $\bta:\bsi \bydef \sum_{i,j}\bta_{ij}\bsi_{ij}$ denotes the tensor product
and 
$G\colon \bSg\to\R$ and $F\colon L^2(\Omega)^2\to\R$ are 
linear functionals 
given by
\[
		G(\bta) \bydef \int_\Gamma (\bta\bn)\cdot\boldsymbol{g}
		\qand
		F(\bfv) \bydef -\int_\Omega \bff\cdot\bfv
\]
for all $\bsi, \bta\in \bSg$ and $\bfv\in L^2(\Omega)^2.$

Furthermore,  \cite[Theorem 2.1]{bbs-2017} states that the solution $(\bsi, \bfu)$ of \eqref{stokes-eq2-weak} is unique and stable with respect to the data, \emph{i.e.}, there exists a constant $C>0$ independent of the solution such that
\begin{equation}
		C \norm{(\bsi, \bfu)}_{H\times [L^2(\Omega)]^2} \leq \norm{\bff}_{[L^2(\Omega)]^2} + 
		\norm{\bg}_{H^{1/2}(\Gamma)}\,.
		\label{bound_continuityStokes}
\end{equation}

\subsection{The non-conforming weak form and finite elements}
Similar to Section \ref{sec_weakPoisson},
we approximate the tensor $\bsi$ in each component by the $\CRh$ element and $\bfu$ by locally constant functions,
that is, 
we consider the discrete spaces  $\bSg_h$,  $\bSg_{h,0}$ and $\bV_h$ defined by
\[
\begin{aligned}
		\bSg_h &\bydef \set{ \bta \in[L^2(\Omega)]^{2\times 2} \colon \bta\in [\CRh]^{2\times 2}} \,,
		\\[0.5ex]
		\bSg_{h,0} &\bydef \set{\bta \in\bSg_h\colon \int_{\Omega} \tr(\bta)=0} \,,
		\\[0.5ex]
		\bV_h &\bydef \set{\bfv\in [L^2(\Omega)]^2 \colon \left.\bfv\right|_T \in[\PP_0(T)]^2 \quad\forall T\in\cT_h}\,.
\end{aligned}
\]
Hence, integration by parts is done locally and gives
\[
	\int_T  \bta_h:\nabla\bfu 
	= 
	- \int_T  \bdiv(\bta_h) \cdot \bfu 
	+ \int_{\partial T}  (\bta_h \bn) \cdot \bfu  \,.
\]
Using  the same arguments as in Section \ref{sec_weakPoisson} 
yields
\[
	\begin{aligned}
		\sum_{T\in\cT_h}  \int_{\partial T}  (\bta_h \bn) \cdot \bfu
		&=
		\sum_{e\subset\Gamma}  \Big[\int_{e} (\bta_h \bn) \cdot \bg\Big]
		 +
		\sum_{e\in\cE_I}  \int_{e} \jump{\bta_h}\cdot\avg{\bfu} +\avg{\bta_h}\cdot\jump{\bfu}
		\,.
		\end{aligned}
\]
Using  $\bfu\in H^1(\Omega)^2,$
we obtain the weak form of the first identity in \eqref{eqStokes2}:
\begin{equation}
	\label{eqConsistencyStokes}
	a(\bsi, \bta_h)
	+ \int_\Omega \bfu\cdot\bdiv_h(\bta_h)
	= G(\bta_h) + \int_{\cE_I} \jump{\bta_h}\cdot\avg{\bfu} 
	\quad \forall \bta_h\in\bSg_h\,.
\end{equation}
This identity shows the consistency error
which does not appear in the discrete formulation, because
$\avg{\bfu_h}_e\in\R^2$ and $\int_e\jump{\bta_h}=\boldsymbol0.$
However, this error  has to be controlled by jumps just as in the other method.
We will see, that the previous identities impose a significant simplification of the Galerkin scheme presented in \cite{bbs-2017}.

\subsection{The theoretical and practical scheme}\label{S:Theorical_Practical_Scheme}
We now define define two numerical schemes.
Both schemes approximate $\bsi\in H_0$ by $\bsi_h\in\bSg_{h,0},$
but the practical scheme imposes the average-free trace by a Lagrange multiplier and allows test-functions with local support.
The later proven equivalence of the schemes allows to reduce the theory.

Identity \eqref{eqConsistencyStokes}  
and the properties mentioned thereafter,
give  the following discrete (theoretical) scheme:
\textit{Find $(\bsi_h,\bfu_h)\in\bSg_{h,0}\times\bV_h$, such that}
\begin{subequations}\label{stokes-LDG-formT}
		\begin{align}
			\label{stokes-LDG-formT1}
			a_{h}(\bsi_h, \bta) + b_h(\bta, \bfu_h) &= G(\bta) \qquad \forall \bta\in\bSg_{h,0}\\
			\label{stokes-LDG-formT2}
			b_h(\bsi_h, \bfv) &= F(\bfv) \qquad \forall \bfv\in\bV_h
		\end{align}
\end{subequations}
where  the bilinear forms $a_{h} \colon \bSg\times \bSg\to\R$ and $b_h\colon \bSg\times [L^2(\Omega)]^2\to\R$ are defined by
\[
	\begin{aligned}
	a_{h}(\bsi,\bta) &\bydef \frac{1}{\nu}\int_{\Omega} \bsi^d :\bta^d + \frac{1}{\nu}\int_{\cE_I}\gamma\jump{\bsi}\cdot\jump{\bta} && \forall \bsi, \bta\in \bSg\,,\\
   b_h(\bta,\bfv) &\bydef \int_{\Omega} \bfv\cdot\bdiv_h(\bta)
   && \forall \bfv\in L^2(\Omega)^2\,,
\end{aligned}
 \]
 with the parameters $\left.\gamma\right|_e = \frac1{\abs{e}}$ 
 and $G$ and $F$ are defined as for \eqref{stokes-eq2-weak}.

The practical scheme reads:
\textit{Find $(\bsi_h, \bfu_h, \phi)\in \bSg_h\times\bV_h\times\R$, such that}
\begin{subequations}
		\label{eq_practicalScheme}
		\begin{align}
		\label{eq_practicalScheme1}
		\Bnc{\bsi_h, \bfu_h}{\bta, \bfv} + \frac{1}{\nu}\phi\int_{\Omega}\tr(\bta) &= G(\bta) +F(\bfv) && \forall (\bta,\bfv)\in\bSg_h\times\bV_h \\
			\frac{1}{\nu}\psi\int_{\Omega}\tr(\bsi_h) &= 0  &&\forall \psi\in\R 
		\label{eq_practicalScheme2}
		\end{align}
\end{subequations}
where the bilinear form $\onlyBnc\colon (\bSg\times\bV_h)\times(\bSg\times\bV_h) \to \R$ is given by
\begin{equation}
 	\Bnc{\bsi, \bfu}{\bta, \bfv} \bydef
	a_{h}(\bsi, \bta) +b_h(\bta,\bfu) + b_h(\bsi,\bfv) \,.
	\label{defBformStokes}
\end{equation}

Given  the property  $\bSg_h = \bSg_{h,0}\oplus \R\I,$
the following result relates the solutions of \eqref{stokes-LDG-formT} and \eqref{eq_practicalScheme}.
\begin{lemma}\label{lem_equivalence}
		The pair  $(\bsi_h, \bfu_h)\in \bSg_{0,h}\times\bV_h$ is a solution of \eqref{stokes-LDG-formT} if
		and only if 
		$(\bsi_h, \bfu_h, 0)\in\bSg_{h}\times\bV_h\times\R$  solves \eqref{eq_practicalScheme}.
\end{lemma}
\begin{proof}
	We proceed similar to \cite[Theorem 4.1]{bbs-2017}.
First, we take a solution of  \eqref{eq_practicalScheme}.
The identity  \eqref{eq_practicalScheme2} gives
$\int_\Omega \tr(\bsi_h)=0$ which implies $\bsi_h\in\bSg_{h,0}$
		and, since $\bSg_{h,0}\subset\bSg_h,$ the identity \eqref{eq_practicalScheme1} gives \eqref{stokes-LDG-formT}.
		We conclude that $(\bsi_h, \bfu_h)$ solves \eqref{stokes-LDG-formT}.
		
	\smallskip
		Second, we introduce a few identities to prove that  every solution of \eqref{eq_practicalScheme} satisfies $\phi=0.$ 
		To this end, 
we write each $\bta\in\bSg_h$ as $\bta = \bta_0 + \rho\I$ with $\bta_0\in\bSg_{h,0}$ and $\rho=\frac1{2\abs{\Omega}}\int_\Omega\tr(\bta).$
Using this decomposition we conclude
		\[
			G(\bta) = G(\bta_0) 
			\quad \text{since}\quad
			G(\I) = \int_{\Gamma} \bfg\cdot\bn = \int_{\Omega} \div\bfu = 0\,.
		\]
	Furthermore, considering definitions we realise that 
		\begin{equation}
				\Bnc{\bta, \bfv}{\I, \boldsymbol{0}} = 0 \quad \forall (\bta,\bfv)\in \bSg_{h}\times\bV_h \,.
		  \label{eq_IisOrthogonal}
		\end{equation}
		These identities and \eqref{eq_practicalScheme1} with the test-pair $(\tilde\bta =\I,\, \bfv=\boldsymbol0)$ prove $2\phi \abs{\Omega} =0$ and $\phi = 0.$
Furthermore,  the previous identities show  that \eqref{eq_practicalScheme1} is equivalent to
		\[
				\begin{aligned}
						F(\bfv) +G(\bta_0) - \Bnc{\bsi_h,\bfu_h}{\bta_0,\bfv}
						= \frac\phi\nu\int_{\Omega}\tr(\bta) \,.
				\end{aligned}
		\]

	Finally, taking a solution $(\bsi_h, \bfu_h)$ of \eqref{stokes-LDG-formT}
		the last equivalence shows that   $(\bsi_h, \bfu_h, 0)$ solves  \eqref{eq_practicalScheme1}.
		Equation \eqref{eq_practicalScheme2} is satisfied by definition since $\bsi_h\in\bSg_{h,0}.$
		This finishes the proof.
\end{proof}

Lemma~\ref{lem_equivalence} proves  that both schemes give the same solutions
and, therefore, we can prove the existence and uniqueness for the scheme \eqref{stokes-LDG-formT} only.

\subsection{Existence and uniqueness}
We aim to prove an inf-sup condition for scheme \eqref{stokes-LDG-formT} which confirms its existence and uniqueness.
To this end, we define the following semi-norms and norms:
\begin{align}
		\label{defAnormStokes}
		\aDGnorm{\bta}^2 &\bydef a_{h}(\bta,\bta) 
		= \frac{1}{\nu}\bLnorm{\bta^d}{\Omega}^2 + \frac{1}{\nu}\bLnorm{\gamma^{1/2}\jump{\bta}}{\cE_I}^2  && \forall\,\bta\in\bSg_0\,,
  \\
  \nonumber
  \bsignorm{\bta}^2 &\bydef \aDGnorm{\bta}^2 + \frac1\nu\Lnorm{\bdiv_h\bta}{\Omega}^2 
  && \forall\,\bta\in\bSg_0\,,
\end{align}
and
\[
  \normNC{(\bta,\bfv)}^2 \bydef \bsignorm{\bta}^2 + \nu\Lnorm{\bfv}{\Omega}^2 \quad\forall\,(\bm{\tau},\bfv)\in\bSg_0\times [L^2(\Omega)]^2\,.
\]

\begin{remark}
It is not immediate  that  $\bsignorm{\cdot}$ is a norm.
Luckily, Lemma 3.10 in \cite{bbs-2017} proves 
the existence of a constant $C > 0,$ independent of the meshsize, such that
\begin{equation}
	C \Lnorm{\bta}{\Omega}^2
	\leq \bLnorm{\bta^d}{\Omega}^2
	+  \Lnorm{\bdiv_h\bta}{\Omega}^2 
	+  \bLnorm{\gamma^{1/2}\jump{\bta}}{\cE_I}^2 
	\qquad\forall\bta\in\bSg_0\,.
	\label{eqLem310BBS17}
\end{equation}
	Hence, $\bsignorm{\cdot}$ is a norm on $\bSg_0$
	Furthermore, since  $\bsignorm{\I}=0$ it is clear that this is not a norm on $\bSg.$
\end{remark}

\begin{lemma}\label{lem_infsupStokes} 
	There exists a constant $C_{sup}>0$ independent of mesh and data, such that 
for every $\bfv\in \bV_h,$
there exists $\tilde\bta \in \bSg_{h,0}$ such that 
\[
\int_\Omega \bdiv_h(\tilde\bta) \cdot\bfv = \Lnorm{\bfv}{\Omega}^2 
\qquad\text{and}\qquad
\bsignorm{\tilde\bta} \leq C_{sup}\,\nu^{-1/2}\Lnorm{\bfv}{\Omega}\,.
\]
\end{lemma}
\begin{proof}
	For $\tilde\bta\in\bSg_h,$ the proof is analogous to the one of Lemma~\ref{lem_infsup2}, 
	since we may consider the separate independent cases $\bfv = (v_1,0)$ and $\bfv= (0,v_2).$
	Then, considering $\bdiv_h\bta = \bdiv_h\bta_0$ and $\bsignorm{\bta}=\bsignorm{\bta_0}$ finishes the proof.
\end{proof}

\begin{theorem}\label{thm_iscStokes} 
Let $\onlyBnc$ be the   bilinearform in \eqref{defBformStokes}.
There exists a constant $\alpha>0,$ such that for all $(\bsi, u)\in\bSigma_h\times V_h$ we have
\begin{equation}
	 \alpha \normNC{(\bsi,\bfu)} \leq
	 \sup_{\bzero\neq(\bta,\bfv)\in\bSg_{h,0}\times \bV_h  }
  \frac{\Bnc{\bsi,\bfu}{\bta,\bfv}}{\normNC{(\bta,\bfv)}} \,.
  \label{eq_iscSystemStokes}
\end{equation}
\end{theorem}
\begin{proof}
	We proceed  analogously to Theorem~\ref{thm_iscSistema}.
	Given $\bsi\in\bSg_{h,0}, \bfu\in\bV_h$ 
	and  bilinearform~\eqref{defBformStokes}
	we define 
$\,\bta \bydef \bsi\,$ and $\,\bfv \bydef -\bfu+ \nu^{-1}\bdiv_h(\bsi)\,$
	to obtain
	\[
			B_{nc}[(\bsi, \bfu), (\bta, \bfv)] =  \bsignorm{\bsi}^2\,.
	\]
	After that, $\tilde\bta$ from Lemma~\ref{lem_infsupStokes} with $\delta=\nu/C_{sup}^2,$ gives
\[
		b(\delta\tilde\bta, \bfu) = \delta \Lnorm{\bfu}{\Omega}^2
		\qand
		B_{nc}[(\bsi, \bfu), (\delta\tilde\bta, \boldsymbol0)] \geq -\frac12\aDGnorm{\bsi}^2 + \frac\delta2\Lnorm{\bfu}{\Omega}^2\,.
\]
Together, for $\bta_2 \bydef \bsi+\delta\tilde\bta$ and $\bfv$ as above  we get
\[
	B_{nc}[(\bsi, \bfu), (\bta_2, \bfv)] \geq \frac12\bsignorm{\bsi}^2 + \frac\delta2\Lnorm{\bfu}{\Omega}^2\,.
\]
Finally, since $\bsignorm{\delta\tilde\bta} \leq \delta C_{sup}\nu^{-1/2}\Lnorm{\bfu}{\Omega} = \delta^{1/2}\Lnorm{\bfu}{\Omega}$ the same arguments as in Theorem~\ref{thm_iscSistema} finish the proof.
\end{proof}

\medskip
\subsection{Stability, errors and convergence}
From now on  $(\bsi,\bfu)$ and $(\bsi_h,\bfu_h)$ will be the unique solutions of \eqref{eqStokes2} and \eqref{stokes-LDG-formT}, respectively.
As before, stability requires a bound on the consistency error.
Therefore, we begin providing quasi-consistency identities.
Taking the difference of \eqref{eqConsistencyStokes} and \eqref{stokes-LDG-formT1}, as well as, \eqref{stokes-eq2-weak2} and \eqref{stokes-LDG-formT2} we get:
\begin{subequations}
  \label{consistencyStokes}
\begin{align}
	\label{consistencyStokesA} 
		a_{h}(\bsi-\bsi_h,\bta) + b(\bta,\bfu-\bfu_h)&= \int_{\cE_I}\jump{\bta}\avg{\bfu} 
		&& \forall\; \bta\in  \bSg_{h,0}\,, 
		\\
	\label{consistencyStokesB}
 b(\bsi-\bsi_h,\bfv) &=0 && \forall\; \bfv\in \bV_h\,.
\end{align}%
\end{subequations}

Next we state a few estimates that are proven analogously to those for the Laplacian.
The techniques 
used in Lemma~\ref{lemConsistencyError}
yield the following  bound:
\begin{equation}
		\abs{\int_{\cE_I}\jump{\bta}\avg{\bfu} } \leq C_1h \,\nu^{1/2}\Hseminorm{\bfu}{\Omega} \aDGnorm{\bta} \quad \forall\; \bta\in  \bSg_{h,0}\,. \\
		\label{boundCEStokes}
\end{equation}
The equation \eqref{consistencyStokesB} and the fact that  $\bdiv\bsi_h\in\bV_h$  gives
\[
		\Lnorm{\bdiv\bsi_h}{T}\leq 
		\Lnorm{\bdiv\bsi}{T} =	\Lnorm{\bff}{T} 
		\qand
		\left.\bdiv_h\bsi_h\right|_T = \frac1{\abs{T}} \int_T \bdiv\bsi\,.
\]
Therfore, just as in Lemma~\ref{lem_ORC_divPoisson} we get
\[
\Lnorm{\bdiv(\bsi-\bsi_h)}{T} 
= \inf_{\bfv\in\bV_h} \Lnorm{\bfv - \bdiv\bsi}{T} 
\leq C h_T^s \norm{\bff}_{s,T}\,.
\]
Furthermore, \eqref{eq_iscSystemStokes}, \eqref{consistencyStokesA},  \eqref{consistencyStokesB} and \eqref{bound_continuityStokes}
give the stability bound:
\[
	\begin{aligned}
	\alpha\normNC{(\bsi_h,\bfu_h)} 
	&\leq \normNC{(\bsi,\bfu)} + C_1 h \,\nu^{1/2}\Hseminorm{\bfu}{\Omega}
	\\
	&\leq \normNC{(\bsi,\bfu)} + C_1 h \,\nu^{-1/2}\Lnorm{\bsi^d}{\Omega} &&\text{by \eqref{eqStokes2}}
	\\
	&\leq C \normNC{(\bsi,\bfu)} 
	.
	\end{aligned}
\]

We have in included the viscosity parameter $\nu$ in the energy norm, so that
constants in existence and uniqueness (Theorem \ref{thm_iscStokes}), in the previous result, as well as, 
in the following theorem are independent of $\nu.$
The dependence of the error estimate on $\nu$ is common in many methods but frequently  hidden in constants.

\begin{theorem}\label{thmSmoothStokes}
		Let $\bsi\in H^2(\Omega)^{2\times2}$ and $\bfu\in H^1(\Omega)^2$ be the solution of \eqref{eqStokes2}
	and $(\bsi_h, u_h)\in\bSg_h\times \bV_h$ be the solution of \eqref{stokes-LDG-formT}.
	Then, 
	\[
			\normNC{(\bsi-\bsi_h,\bfu-\bfu_h)} \leq C \,h 
			\left( \nu^{-1/2}\norm{\bsi}_{2,\Omega} + \nu^{1/2}\Hseminorm{\bfu}{\Omega}\right)   \,,
	\]
	where $C>0$ is a constant independent of data and $h,  \bsi$ and $\bfu.$
\end{theorem}
\begin{proof}
		The arguments are analoguous to Theorem~\ref{thmSmoothConvergence},
		\emph{i.e.},
		 we split the error into a discrete error and a projection error:
		\[
				\begin{aligned}
				\bsi_h-\bsi &= e_h^{\bsi} + e^{\bsi} 
				\quad\text{with}\quad
				e_h^{\bsi} \bydef  \bsi_h - \Pi_{\sigma} \bsi 
				\quad\text{and}\quad
				e^{\bsi} \bydef \Pi_{\sigma} \bsi - \bsi \,, \\
				u_h-u &= e_h^{u} + e^{u} 
				\quad\text{with}\quad
				e_h^{u} \bydef  \bfu_h - \Pi \bfu 
				\quad\text{and}\quad
				e^{u} \bydef \Pi \bfu - \bfu \,,
				\end{aligned}
		\]
		where  $\Pi$ is the $L^2$-projection into $V_h$ and $\Pi_\sigma$ is the Crouzeix-Raviart projection.
		Then, the norm allows to bound the errors separately as follows:
		\[
			  \normNC{(\bsi-\bsi_h,\bfu-\bfu_h)} \leq 
			  \normNC{(e_h^{\bsi},e_h^u)} +  \normNC{(e^{\bsi},e^u)} \,.
		\]
		We bound the discrete error (the left norm), followed by the projection error (the right norm).
		The consistency error identity \eqref{eqConsistencyStokes}
		and Theorem~\ref{thm_iscStokes} yield the existence of $\alpha>0,$ such that
\begin{align}
	\nonumber
	\alpha \normNC{(e_h^{\bsi}, e_h^u)} 
	&\leq \normNC{(\bsi-\Pi_\sigma\bsi,\bfu-\Pi \bfu)} + C_1 h \,\nu^{1/2}\Hseminorm{\bfu}{\Omega} \,,
\end{align}
where we bounded the consistency error using \eqref{boundCEStokes}.

From here on, 
each part of the norm $\normNC{\cdot}$ defined just after \eqref{defAnormStokes}
can be bounded as in Theorem~\ref{thmSmoothConvergence}
which ends the proof.
\end{proof}

In summary we have shown that the same analysis works for the Poisson and the Stokes problem both formulated in a dual mixed  non-conforming approach.

\medskip
\clearpage
\section{Numerical confirmation}\label{secNumerics}

In this section, we present numerical experiments that illustrate the performance 
of our  method   and confirm  the convergence rates for  
smooth solutions and 
exhibit  optimal convergence rates for non-smooth solutions that are not supported by our theory.
The numerical experiments were performed with the finite element toolbox ALBERTA using refinement by recursive bisection \cite{Alberta}. 
The solutions of the corresponding linear systems were computed using the backslash operator of MATLAB. 

In all cases, we start with an initial mesh $\mathcal{T}_0,$ and in each step we solve the corresponding problem, compute the errors and  refine uniformly to generate a conforming refinement $\mathcal{T}_{k +1}$ of $\mathcal{T}_k$ by bisecting all elements twice. 
In what follows, $\mbox{DOFs}_k$ denotes the total number of degrees of freedom (unknowns) of the corresponding
system on $\mathcal{T}_k.$
Since we only use uniform meshes we have $\mbox{DOFs}^{-r/2} \sim h^r.$
Hence, we calculate the empirical order of convergence $(\mbox{EOC}),$ associated to some global error $e_k$ in step $k,$ by
\[
	\text{EOC}(e_k) \bydef -\frac{2\, \log(e_k/e_{k-1})}{\log(\text{DOFs}_k/\text{DOFs}_{k-1})}, \qquad k>0. 
\]

In order to impose the continuity of the normal trace across inter-element edges, we always choose the  jump parameter $\left.\gamma\right|_e = \abs{e}^{-1}.$
We present three numerical examples for each of the analyzed  models (Poisson and Stokes). 
The first one exhibits a smooth behavior, completely supported by our theory.
But, since the divergence error equals the best-approximation error for any regularity,
we also analyze  the performance of the methods 
for  a singular solution on an M-shaped domain
and on a crack domain (opening angle zero)
where the crack is resolved by the mesh and where $\bsi\notin H^s(\Omega),\, s\geq1/2.$
The  used opening angle of the crack is zero, but Figure~\ref{figInitialMeshes} shows an open crack to illustrate the ``additional'' DOFs on the ``overlapping edges''.

\begin{figure}[hbt]
	\centering
	\includegraphics[scale=0.875]{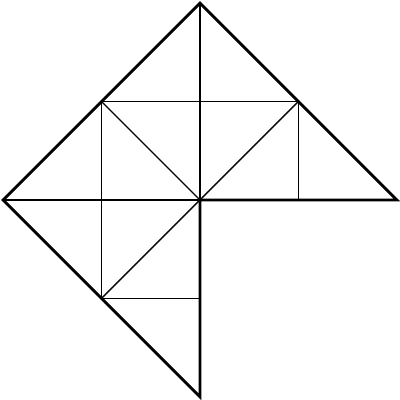}
	\qquad
	\includegraphics[scale=0.875]{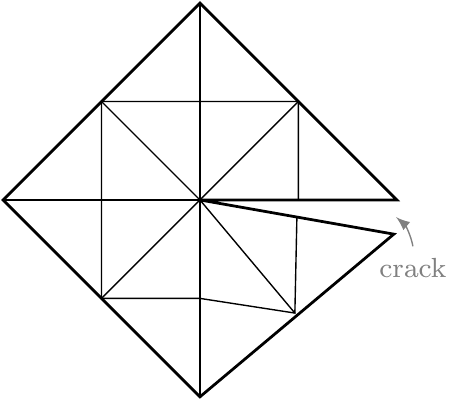}
	\caption{Initial meshes for problems on M-shaped domains and the crack problem (real opening angle is zero).}
	\label{figInitialMeshes}
\end{figure}

Note that  asymptotically  33\%  of the DOFs for $\bsi_h$ are fixed by the best approximation of  $\div \bsi.$
This is quite a big percentage, so we analysed the perfomance of the method for the non-smooth cases 
to see whether the other DOFs are also close to an optimal converging approximation, as the 
smooth convergence result proves.
It turns out, that for all cases we obtain optimal convergence (a result left open to be proven as part of future work).

\clearpage
\subsection{Poisson problems}
\subsubsection{P1. A smooth test solution}
In this first test, we  consider a smooth solution on an M-shaped domain with Dirichlet boundary conditions given by:
\begin{equation}
\begin{split}
		u &= \exp\big(-10(x^2+y^2)\big)
		\;,\quad
		f= -\Delta u
		\;,\quad
		g = \left.u\right|_{\partial\Omega} \,,
		\\
\Omega&=  \{|x| + |y|<1\} \cap \{x<0 \mbox{ or } y>0 \}\,.
		\label{eqTestProblem1}
\end{split}
\end{equation}
The objective of this example is to confirm that the analyzed scheme provides
the optimal rate of convergence.
	The initial mesh on the non-convex domain $\Omega$ is shown  in Figure~\ref{figInitialMeshes} (left).
Additionally, we report the convergence rates in Table~\ref{tabTestP1Laplace} which are in agreement with the theory.
The L2-error of $\bsi_h$ even converges faster
which is an advantage over the classic RT-element.
Finally, in Figure~\ref{figGradSigmaP1} we present the module of $\bsi_h$ and $u_h$ at iterations 2, 4 and 6, respectively. 
The pictures seem to indicate  that our method provides a field $\bsi_h$ with a continuous normal trace  across inter-element edges.

\begin{figure}[h!]
		\centering
		\includegraphics[scale=0.1]{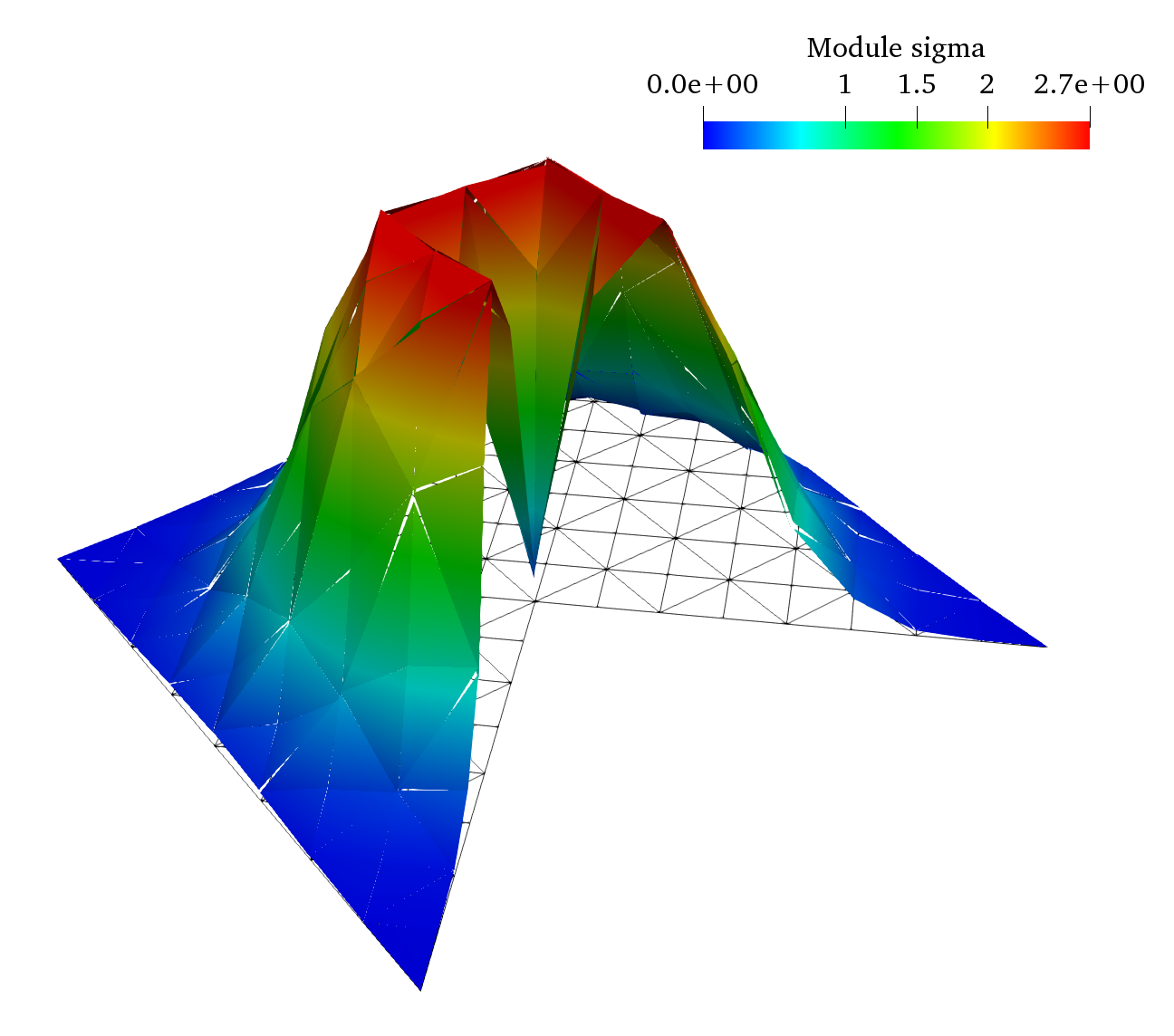}
		\;
		\includegraphics[scale=0.1]{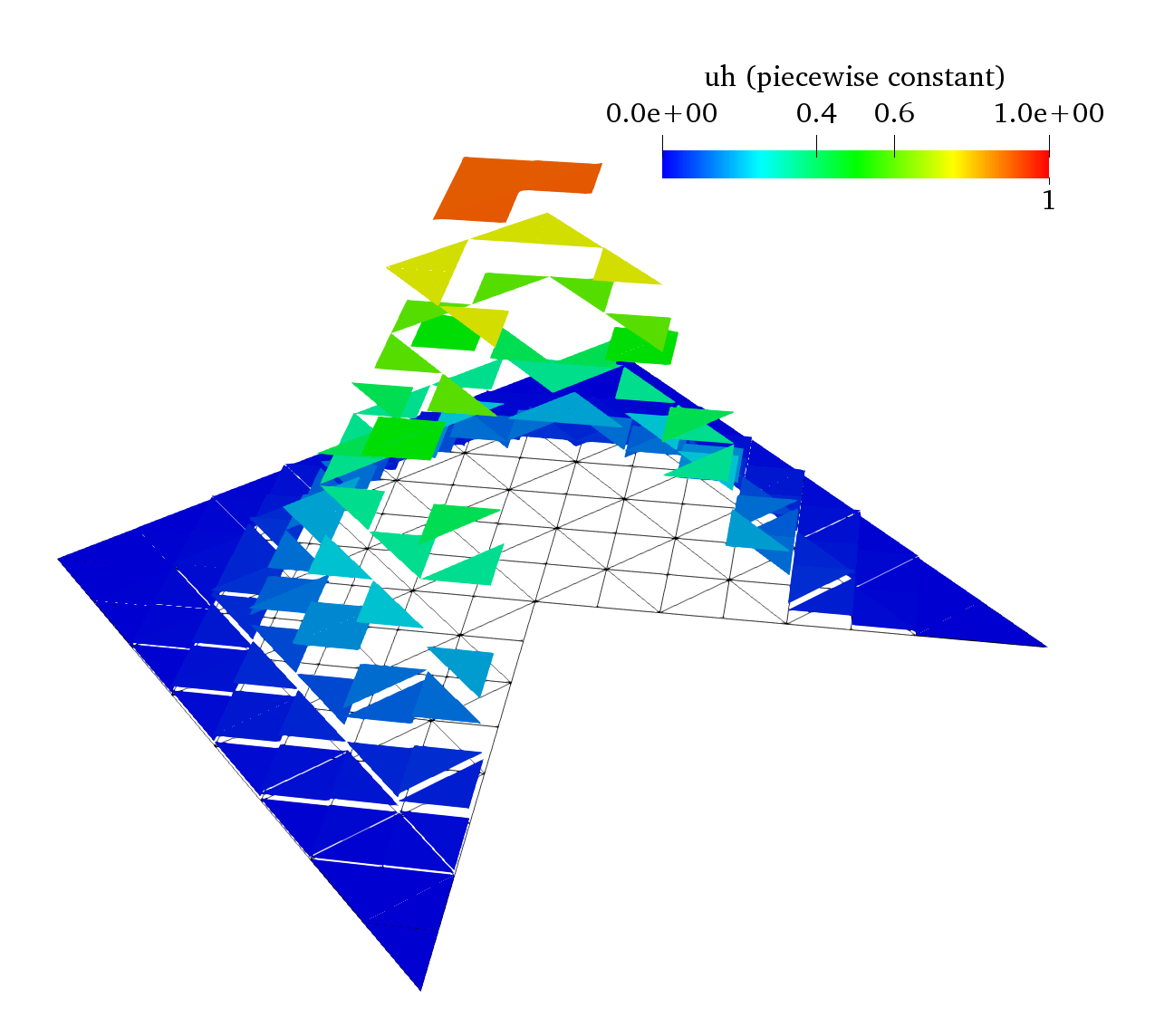}

		\includegraphics[scale=0.1]{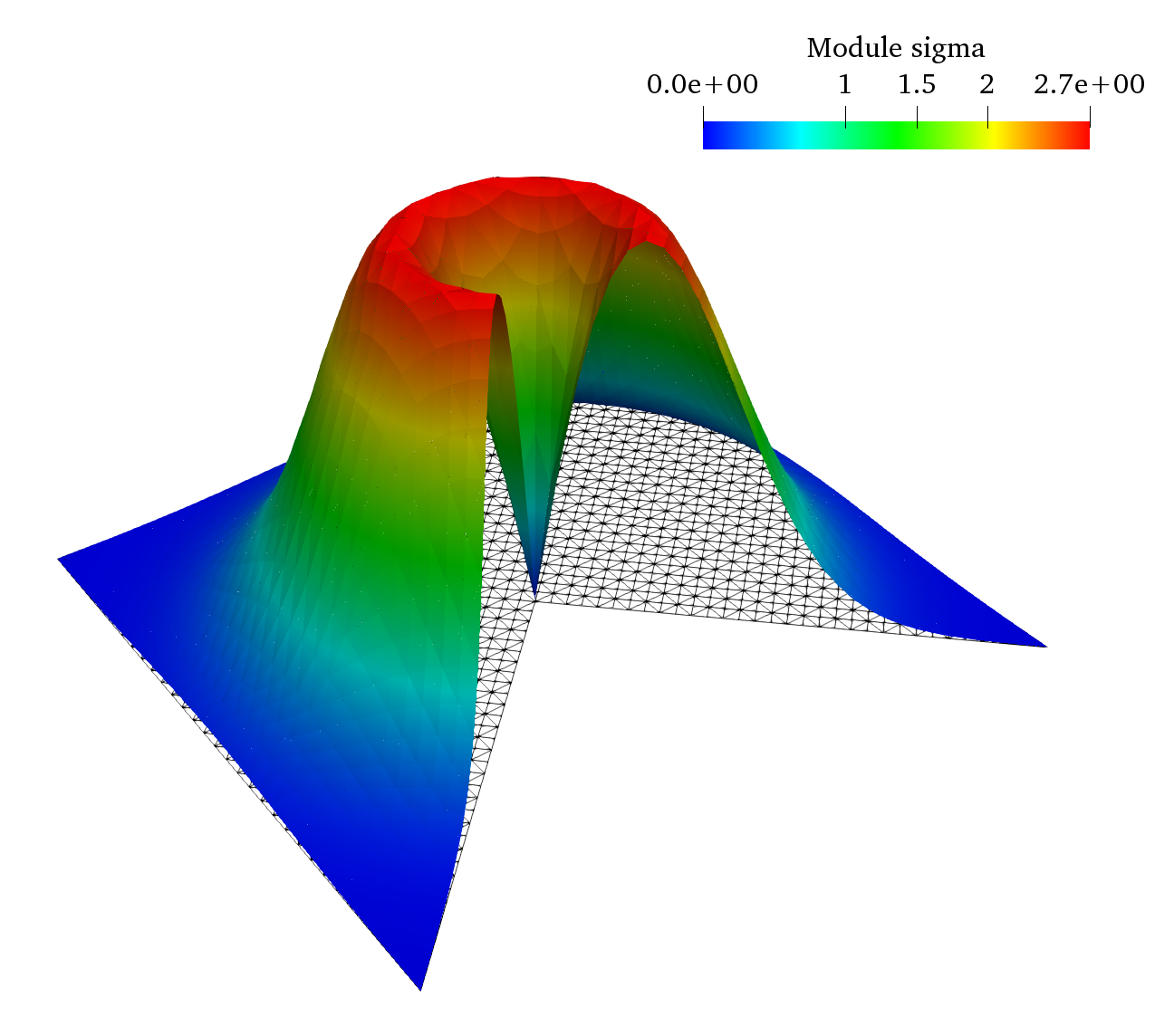}
		\;
		\includegraphics[scale=0.1]{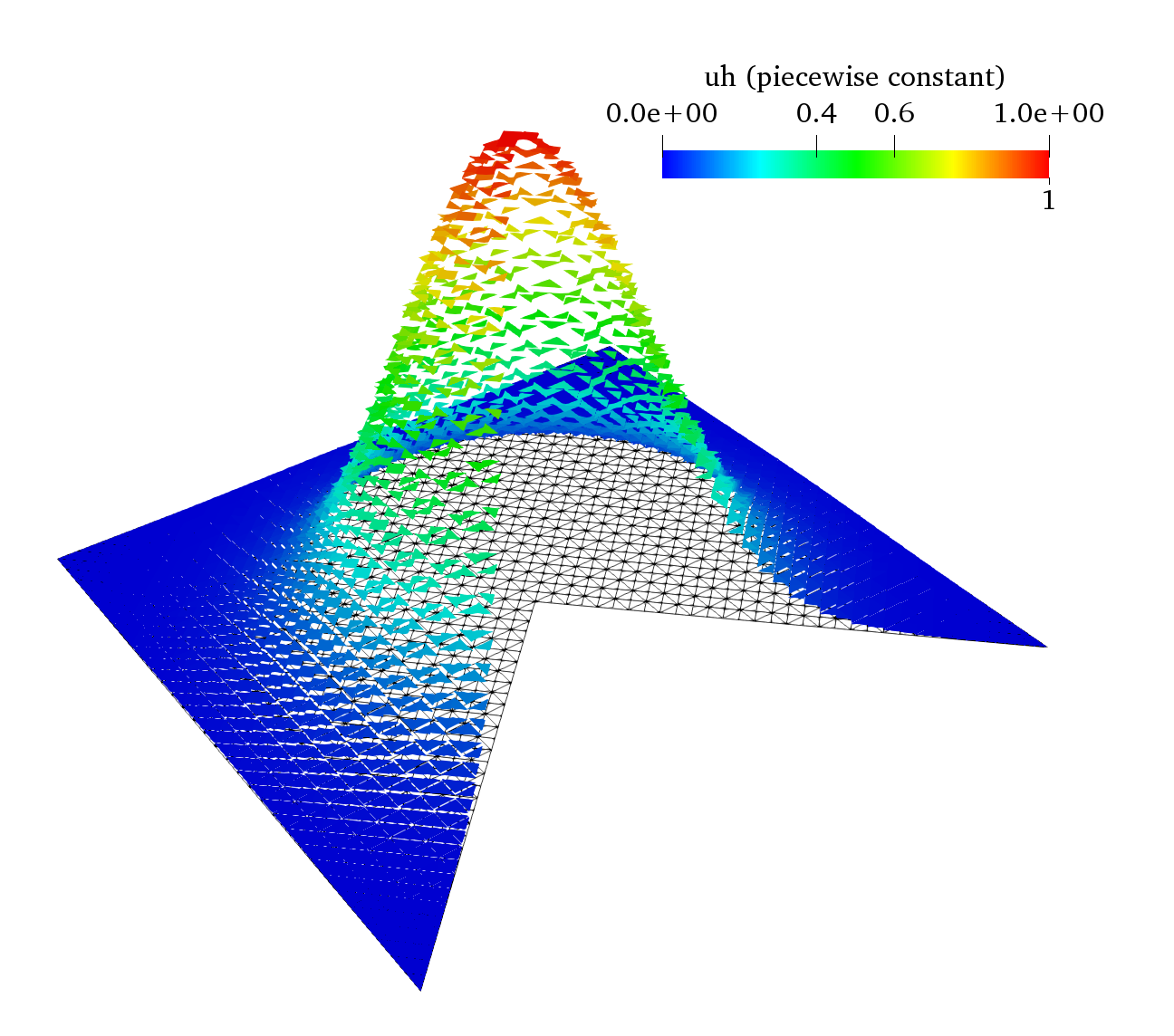}

		\includegraphics[scale=0.1]{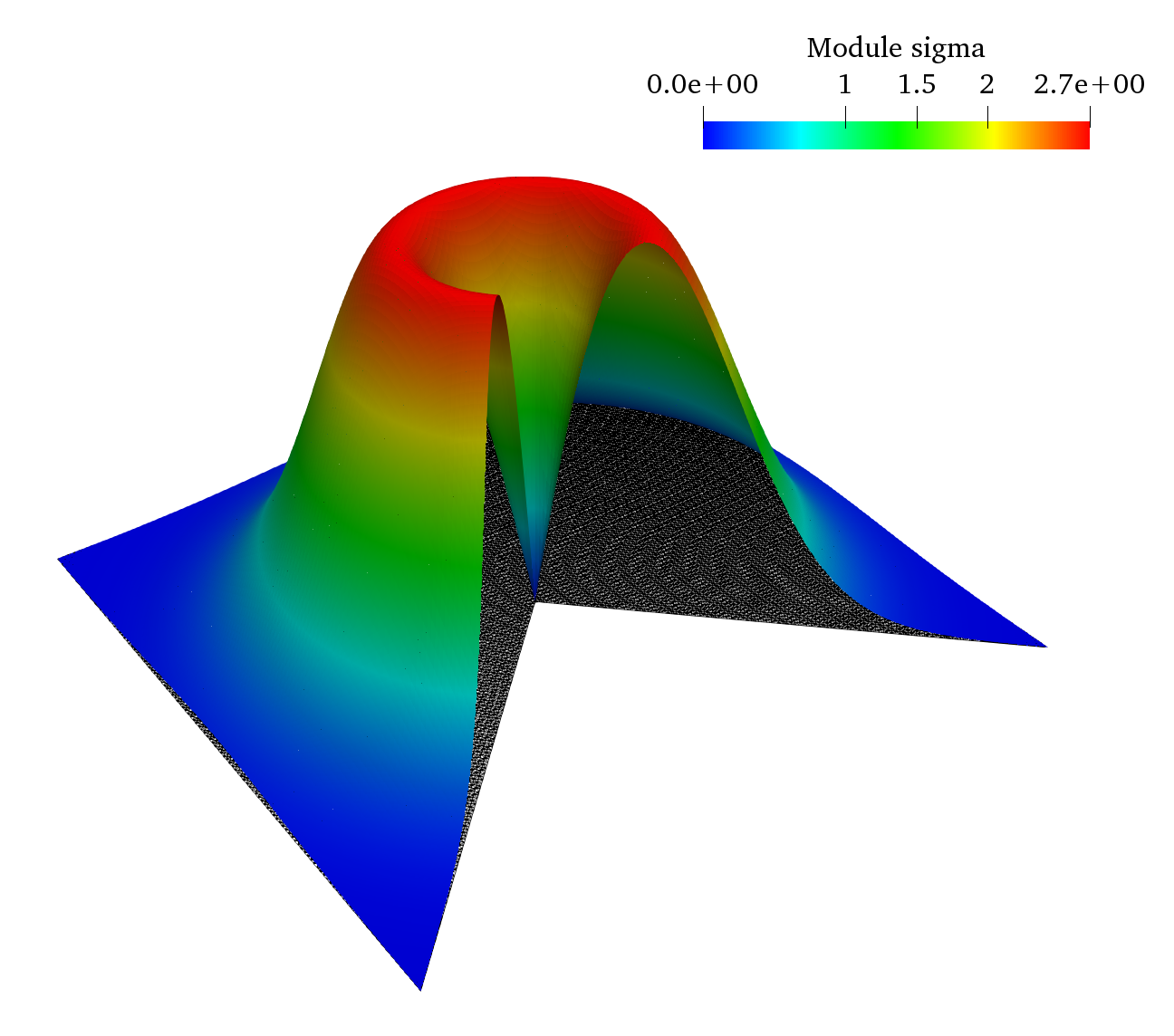}
		\;
		\includegraphics[scale=0.1]{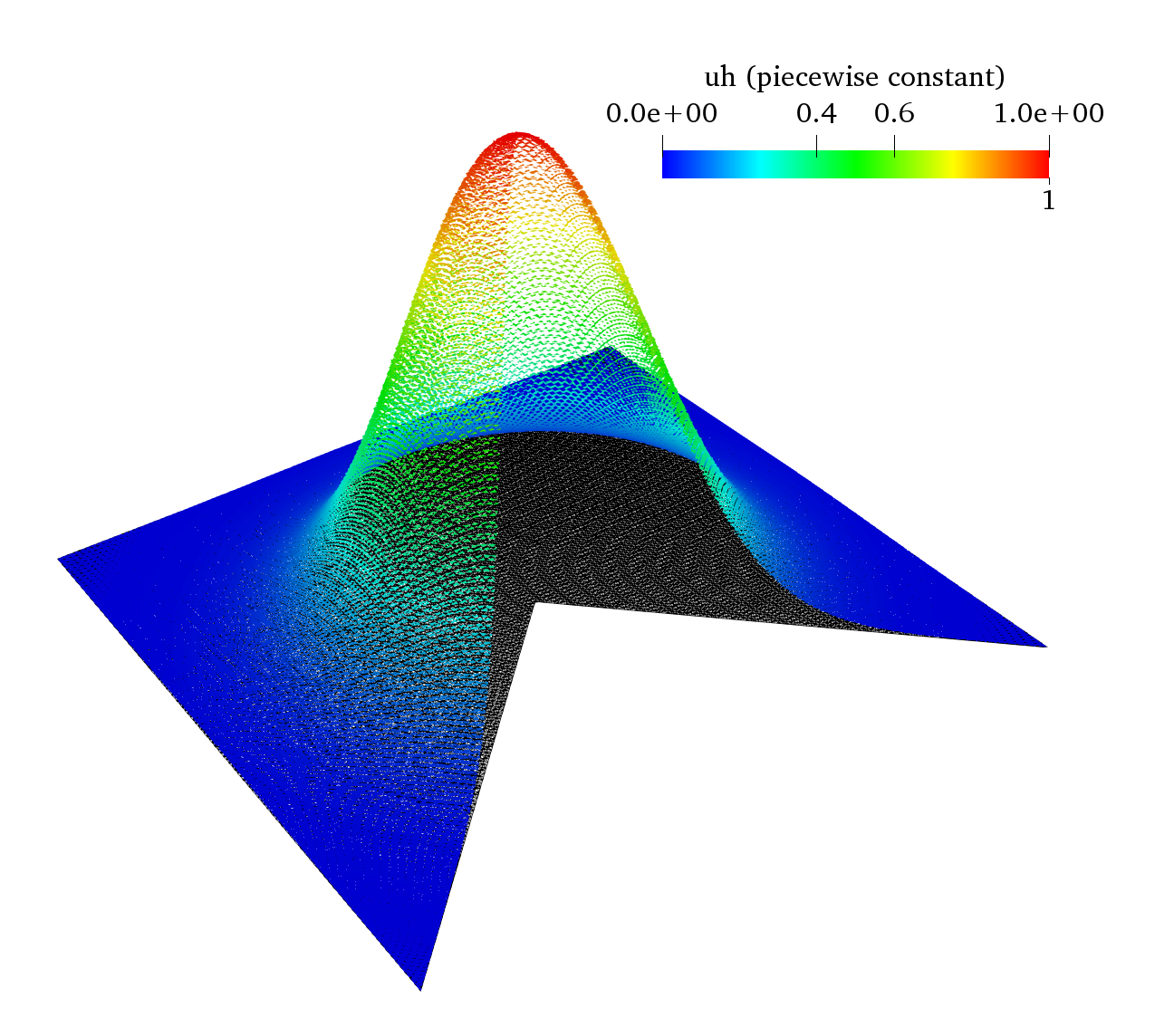}
\caption{P1. Module of $\bsi_h$ (left) and $u_h$ (right) at iterations 2, 4 and 6 (top to bottom).}
		\label{figGradSigmaP1}
\end{figure}

\begin{table}[hbt!]
\centering \scriptsize
\caption{Error behavior for P1 problem \eqref{eqTestProblem1}}%
\label{tabTestP1Laplace}%
\begin{tabular}{rrcccccccc}
\toprule
iter & DOFs 
& $\Lnorm{\bsi-\bsi_h}{\Omega}$ & EOC 
& $\Lnorm{\div(\bsi-\bsi_h)}{\Omega}$& EOC
& $\Lnorm{\gamma^{1/2} \jump{\bsi_h}}{\cE_I}$ & EOC 
& $\Lnorm{u-u_h}{\Omega}$ & EOC \\
\midrule
0  &       58   &          1.067 \hspace{6mm}  &--         &         9.484 \hspace{6mm} &--         &     2.532$\cdot10^{-1}$ & --                &          2.270$\cdot10^{-1}$  & -- \\  
1   &     212      &       3.708$\cdot10^{-1}$   &1.63   &            5.009 \hspace{6mm} & 0.98   &          1.664$\cdot10^{-1}$& 0.65             &            9.509$\cdot10^{-2}$  & 1.34\\ 
2   &     808     &       8.427$\cdot10^{-2}$   &2.21   &           2.175  \hspace{6mm} & 1.25     &        8.967$\cdot10^{-2}$  &0.92              &             4.517$\cdot10^{-2}$   &1.11 \\
3   &    3152    &        2.114$\cdot10^{-2}$   &2.03   &           1.102  \hspace{6mm} &1.00       &      4.693$\cdot10^{-2}$  &0.95                 &           2.261$\cdot10^{-2}$   &1.02 \\
4   &   12448   &        5.276$\cdot10^{-3}$  &2.02   &           5.528$\cdot10^{-1}$  &1.00       &      2.380$\cdot10^{-2}$  &0.99                 &            1.130$\cdot10^{-2}$   &1.01 \\
5   &   49472  &         1.318$\cdot10^{-3}$   &2.01  &            2.766$\cdot10^{-1}$  &1.00 &            1.195$\cdot10^{-2}$ & 1.00        &                   5.652$\cdot10^{-3}$&   1.00 \\
6   &  197248 &         3.295$\cdot10^{-4}$   &2.00  &            1.383$\cdot10^{-1}$&  1.00    &         5.987$\cdot10^{-3}$  &1.00     &                        2.826$\cdot10^{-3}$&   1.00 \\
7   &  787712&          8.239$\cdot10^{-5}$   &2.00  &              6.917$\cdot10^{-2}$  &1.00  &           2.995$\cdot10^{-3}$  &1.00   &                     1.413$\cdot10^{-3}$   & 1.00 \\
\bottomrule
\end{tabular}
\end{table}

\subsubsection{P2. M-Shaped domain}
Let $\Omega =\{|x| + |y|<1\} \cap \{x<0 \mbox{ or } y>0 \}$ be the domain shown in Figure~\ref{figInitialMeshes} (left), and consider the following solution of \eqref{model}:
\begin{equation}
		u(r,\theta)  = r^{\frac23} 
		\sin\left(\frac{2\theta}{3}\right) - \frac{r^2}{4},
		\qquad
		g= u_{|\partial \Omega}
		\label{eqP2}
\end{equation}
where $(r,\theta)$ denote polar coordinates. 
Note that $u$ is the solution of the elliptic equation $-\Delta u =1$ in $\Omega,$ and that  $u\in H^{5/3}(\Omega)$ (we ignore the $-\varepsilon$).
Figure~\ref{figUSP2} shows the discrete solution $u_h$ (piecewise constant) and the module of $\bsi_h$ (module of a CR field) at iteration 6, respectively. 
We appreciate some small oscillations in the second picture, due to the singularity, but this does not  affect the rates reported in Table~\ref{tabTestP2Laplace}.
We omit the $\div_h(\bsi - \bsi_h)$ error, because it is almost zero ($f=1$ is constant).
Since  $\bsi=\nabla u\in H^{2/3}(\Omega)$ at the reentrant corner,  we expect (and see) the rate $h^{2/3} \sim \mbox{DOFs}^{-1/3}$ for the $L_2$-error of $\bsi.$
The jump terms converge faster and the $L_2$-error of $u$ converges with the optimal rate $h \sim \mbox{DOFs}^{-1/2}.$

\begin{table}[ht!]
\centering
\caption{Error behavior for P2  problem \eqref{eqP2}}
\label{tabTestP2Laplace}
\scriptsize%
\begin{tabular}{rrcccccc}
\toprule
iter & DOFs 
& $\Lnorm{\bsi-\bsi_h}{\Omega}$ & EOC
& $\Lnorm{\gamma^{1/2}\jump{\bsi_h}}{\cE_I}$ & EOC 
& $\Lnorm{u-u_h}{\Omega}$ & EOC \\
\midrule
1 &       212 & 1.152$\cdot10^{-1}$ &  0.68 &   1.033$\cdot10^{-1}$  &0.66   & 5.465$\cdot10^{-2}$ &  1.04\\
2 &       808 & 7.524$\cdot10^{-2}$ &  0.64   & 5.610$\cdot10^{-2}$  &0.91   & 2.750$\cdot10^{-2}$ &  1.03\\
3 &      3152 & 4.849$\cdot10^{-2}$ &  0.65   & 2.889$\cdot10^{-2}$  &0.97   & 1.377$\cdot10^{-2}$ &  1.02\\
4 &     12448 & 3.092$\cdot10^{-2}$ &  0.66 &  1.461$\cdot10^{-2}$  &0.99   & 6.886$\cdot10^{-3}$ &  1.01\\
5 &     49472 & 1.960$\cdot10^{-2}$ &  0.66  & 7.331$\cdot10^{-3}$  &1.00   & 3.442$\cdot10^{-3}$ &  1.01\\
6 &    197248 & 1.238$\cdot10^{-2}$  & 0.66 & 3.669$\cdot10^{-3}$  &1.00   & 1.721$\cdot10^{-3}$ &  1.00\\
7 &    787712 & 7.814$\cdot10^{-3}$ &  0.67 & 1.834$\cdot10^{-3}$  &1.00   & 8.601$\cdot10^{-4}$ &  1.00\\
\bottomrule
\end{tabular}
\end{table}

\begin{figure}[hbt!]
		\centering
		\includegraphics[scale=0.11]{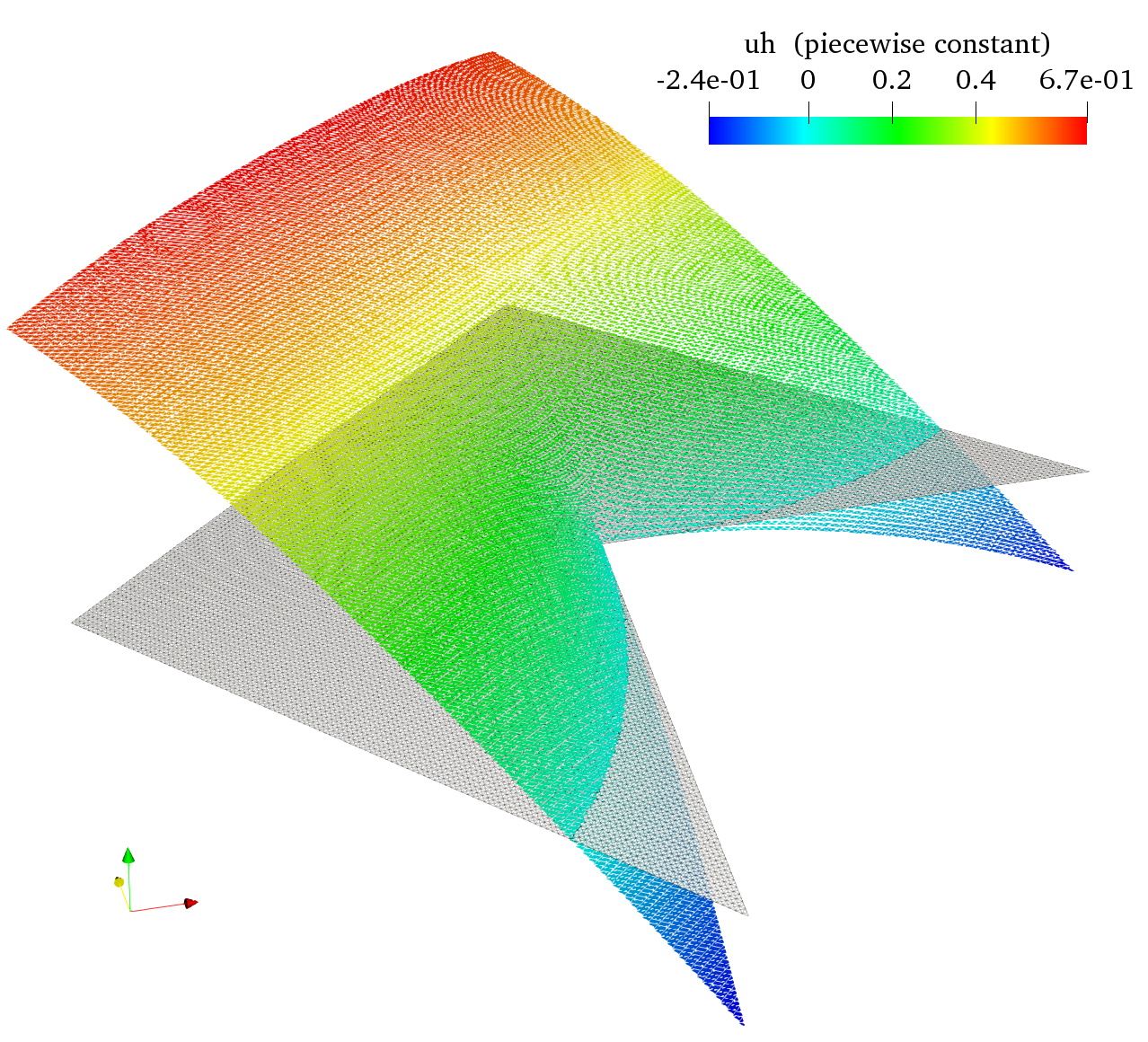}
		\qquad
		\includegraphics[scale=0.11]{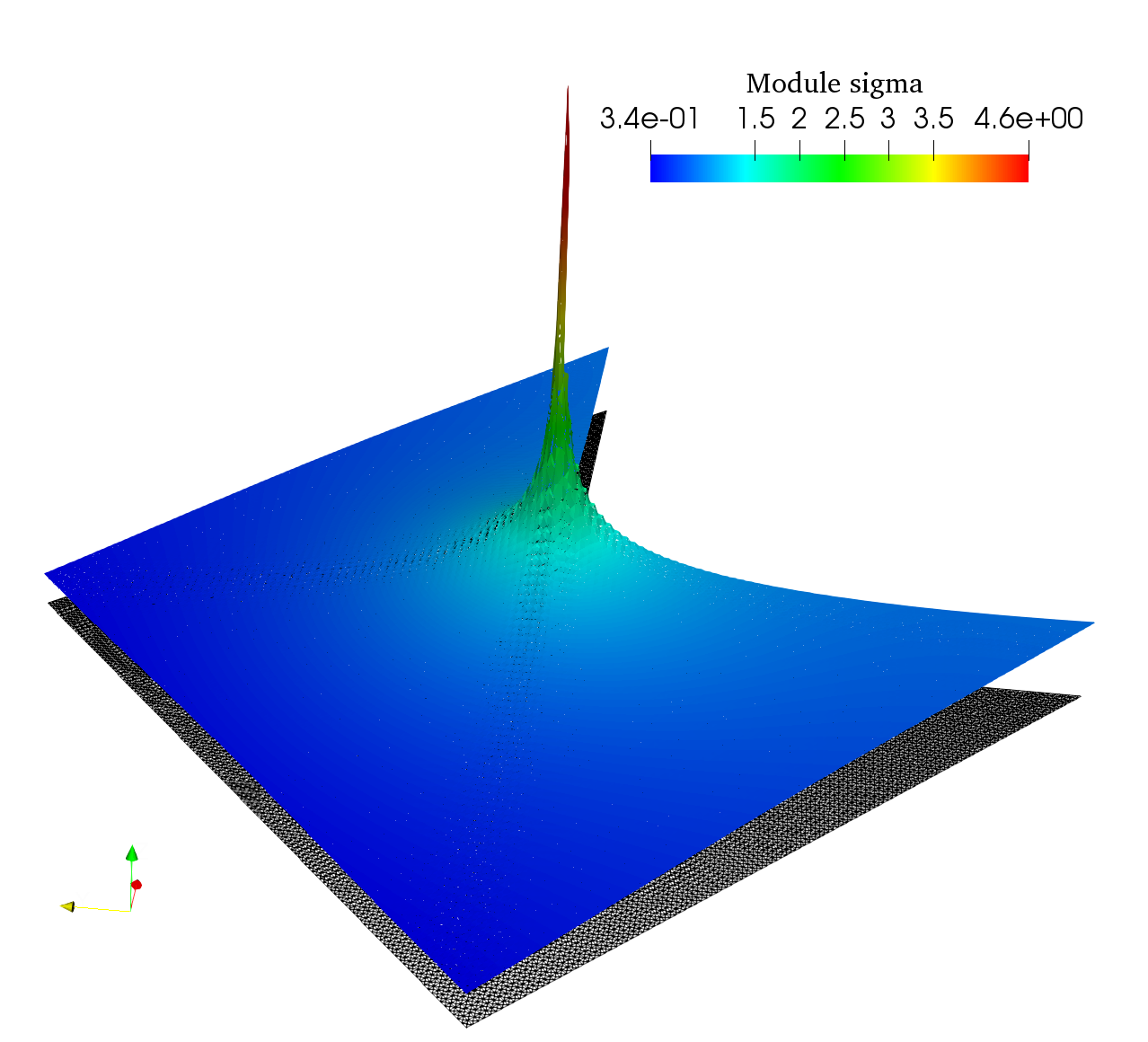}
\caption{P2. $u_h$ (left) and module of $\bsi_h$ (right) at iteration 6.}
		\label{figUSP2}
\end{figure}

\FloatBarrier
\subsubsection{P3. Crack domain} \label{S:P3}

Let $\Omega =\{|x| + |y|<1\} \setminus \{0 \leq x \leq 1, y = 0 \}$ be the crack domain shown in Figure~\ref{figInitialMeshes} (right) with opening angle zero, 
and consider the following singular solution of \eqref{model}:
\begin{equation}
		u(r,\theta)  = r^{\frac12} 
		\sin\left(\frac{\theta}{2}\right),
		\qquad
		g= u_{|\partial \Omega}
		\label{eqP3}
\end{equation}
where $(r,\theta)$ denote polar coordinates. 
Note that $u$ is the solution of the elliptic equation $\-\Delta u =0$ in $\Omega.$
The initial mesh is shown in Figure~\ref{figInitialMeshes} (right).
We emphasize that the DOFs on the line $(0,1) \times\set{0}$ are different.

The error behaviours for this test are reported in Table~\ref{tabTestP3Laplace}.
 We observe optimal rates for the error of  variable $u,$ and the expected rate for the error of $\bsi$ (according to the singularity).
Finally, in Figure~\ref{figUMP3} we show the discrete solution $u_h$ and the module of $\bsi_h$ at iteration  5.

\begin{table}[hbt!] \scriptsize
\centering
\caption{Error behavior for P3  problem \eqref{eqP3}}
\label{tabTestP3Laplace}
\begin{tabular}{rrcccccc}
\toprule
iter & DOFs 
& $\Lnorm{\bsi-\bsi_h}{\Omega}$ & EOC 
& $\Lnorm{\gamma^{1/2}\jump{\bsi_h}}{\cE_I}$ & EOC
& $\Lnorm{u-u_h}{\Omega}$ & EOC\\
\midrule
0 &       76 & 3.410$\cdot10^{-1}$ &   --  & 1.968$\cdot10^{-1}$  &  --  & 1.290$\cdot10^{-1}$ &   --\\
1 &      280 & 2.648$\cdot10^{-1}$ &  0.39  & 1.353$\cdot10^{-1}$  & 0.57  & 6.817$\cdot10^{-2}$ &  0.98\\
2 &     1072 & 2.078$\cdot10^{-1}$ &  0.36  & 8.060$\cdot10^{-2}$  & 0.77  & 3.561$\cdot10^{-2}$ &  0.97\\
3 &     4192 & 1.581$\cdot10^{-1}$ &  0.40 & 4.592$\cdot10^{-2}$  & 0.82  & 1.833$\cdot10^{-2}$ &  0.97\\
4 &    16576 & 1.176$\cdot10^{-1}$ &  0.43  & 2.602$\cdot10^{-2}$  & 0.83  & 9.344$\cdot10^{-3}$ &  0.98\\
5 &    65920 & 8.598$\cdot10^{-2}$ &  0.45  & 1.482$\cdot10^{-2}$  & 0.82  & 4.734$\cdot10^{-3}$ &  0.98\\
6 &   262912 & 6.222$\cdot10^{-2}$ &  0.47 &   8.505$\cdot10^{-3}$  & 0.80  & 2.389$\cdot10^{-3}$ &  0.99\\
7 &  1050112 & 4.470$\cdot10^{-2}$ &  0.48 &    4.921$\cdot10^{-3}$  & 0.79  & 1.202$\cdot10^{-3}$ &  0.99\\
\bottomrule
\end{tabular}
\end{table}

\begin{figure}[hbt!]
		\centering
		\includegraphics[scale=0.11]{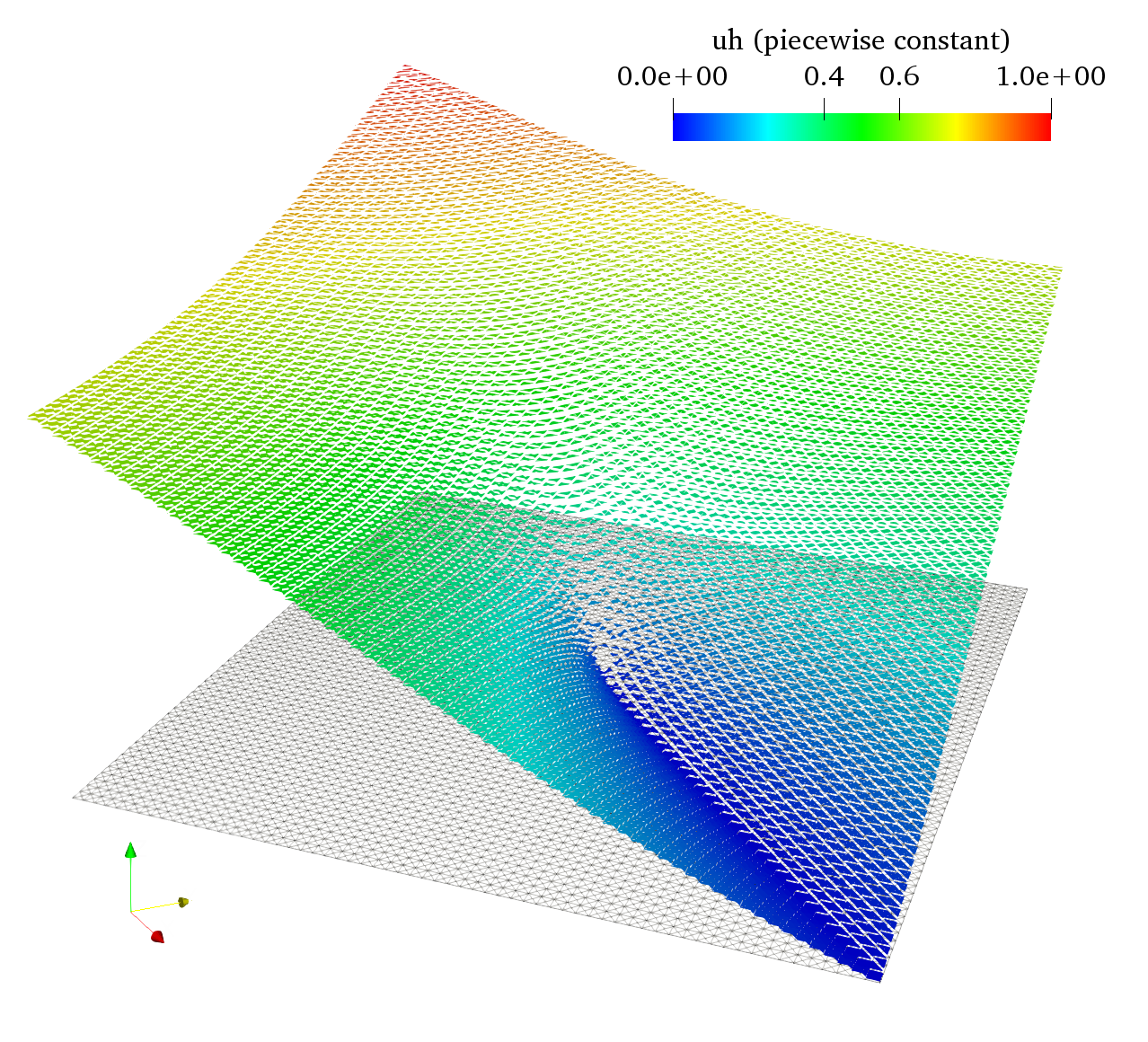}
		\qquad
		\includegraphics[scale=0.11]{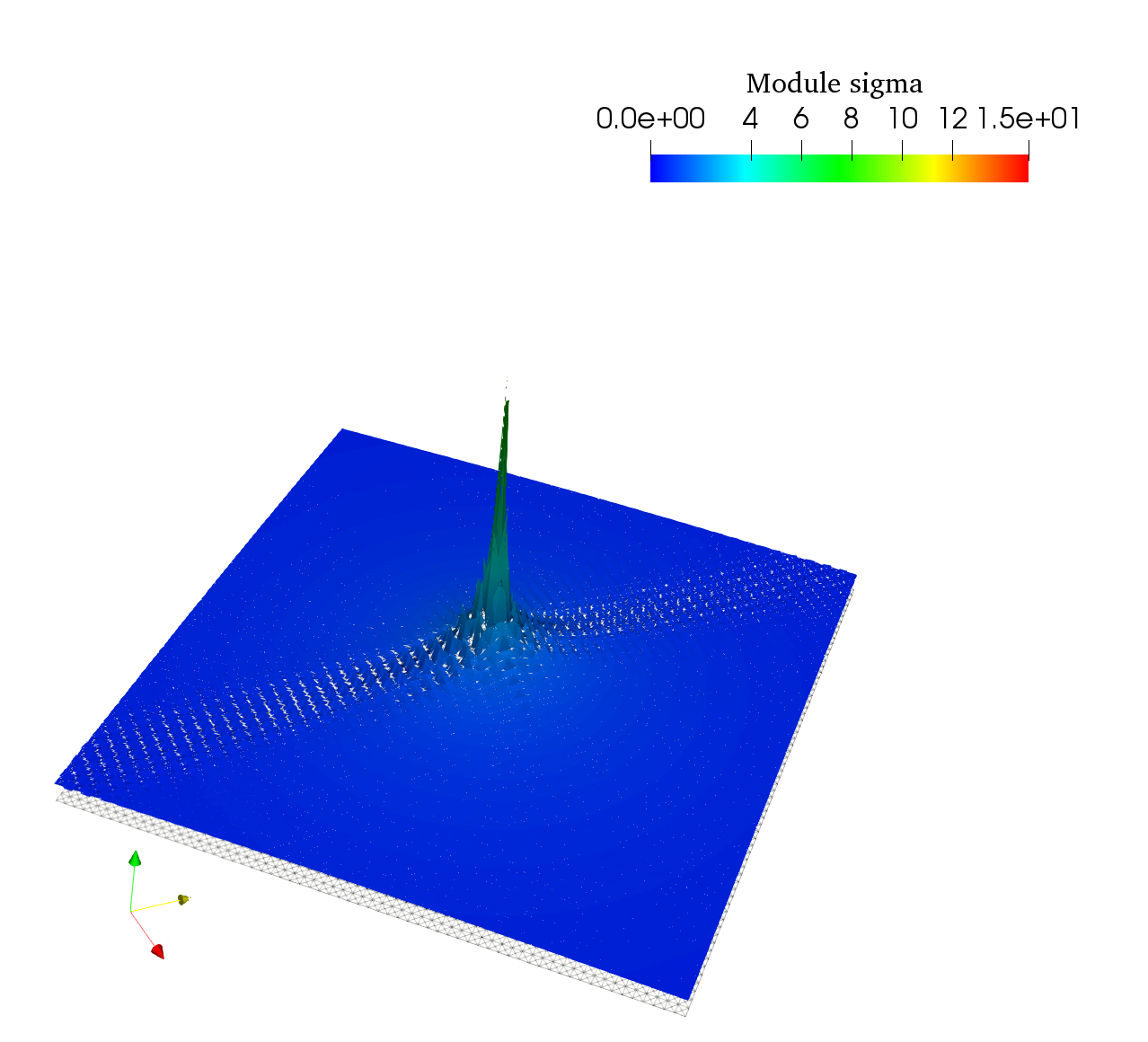}
		\caption{P3. $u_h$ (left) and module of $\bsi_h$ (right) at iteration 5.}
		\label{figUMP3}
\end{figure}

\FloatBarrier
\subsection{Stokes problem}
In order to impose the  zero mean value condition for the trace of functions in the spaces $\bSg_{h,0},$  we consider the equivalent problem \eqref{eq_practicalScheme}, that was introduced in Section
\ref{S:Theorical_Practical_Scheme}.
We recall that from definition of pseudostress $\bsi,$ the pressure can be recovered as:
$p = -\frac{1}{2} \tr(\bsi).$
This identity allows us to analyze the pressure error, since
\[
	\| p - p_h \|_{0,\Omega } = \frac{1}{2} \|\tr{\bsi} - \tr{\bsi}_h \|_{0,\Omega }
	\leq \frac{\sqrt{2}}{2} \| \bsi  - \bsi_h \|_{0,\Omega }\,.
\]

\subsubsection{S1. Smooth solution} 
This example is motivated by the two-dimensional analytical solution of the
Navier-Stokes equations derived by Kovasznay in \cite{Kovasznay1948}, where
we consider $\Omega = \big(-\frac12,\frac32\big)\times (0,2)$ and:
\begin{equation}\label{eqS1}
	\boldsymbol{u}(x,y) = \left( 
\begin{array}{c}
1 - e^{\lambda x} \cos ( 2 \pi y) \\ \frac{\lambda}{2\pi} e^{\lambda x} \sin (2 \pi y)
\end{array}
\right), 
\quad
p(x,y) = -\frac{1}{2} e^{2\lambda x} - p_0,  
\end{equation}
with the parameter $\lambda$ given by:
$\lambda = - \displaystyle\frac{8 \pi^2}{ \nu^{-1} + \sqrt{\nu^{-2} + 16 \pi^2}}.$
In order to test the robustness of our scheme, we solve the problem for different values of 
the viscosity, ranging from $1$ to $10^{-5}.$
Figure~\ref{figS1_EOC} shows 
the decay of the errors versus the DOFs. 
The convergence rates 
confirm
the theory (Theorem \ref{thmSmoothStokes}) in all cases, that is, independent of the value of the viscosity we at least have linear convergence. 
We remark that the pressure errors decay faster with rate $h^2 = \mbox{DOFs}^{-1}.$

\begin{figure}[hbt!]
		\centering
		\includegraphics[scale=0.3]{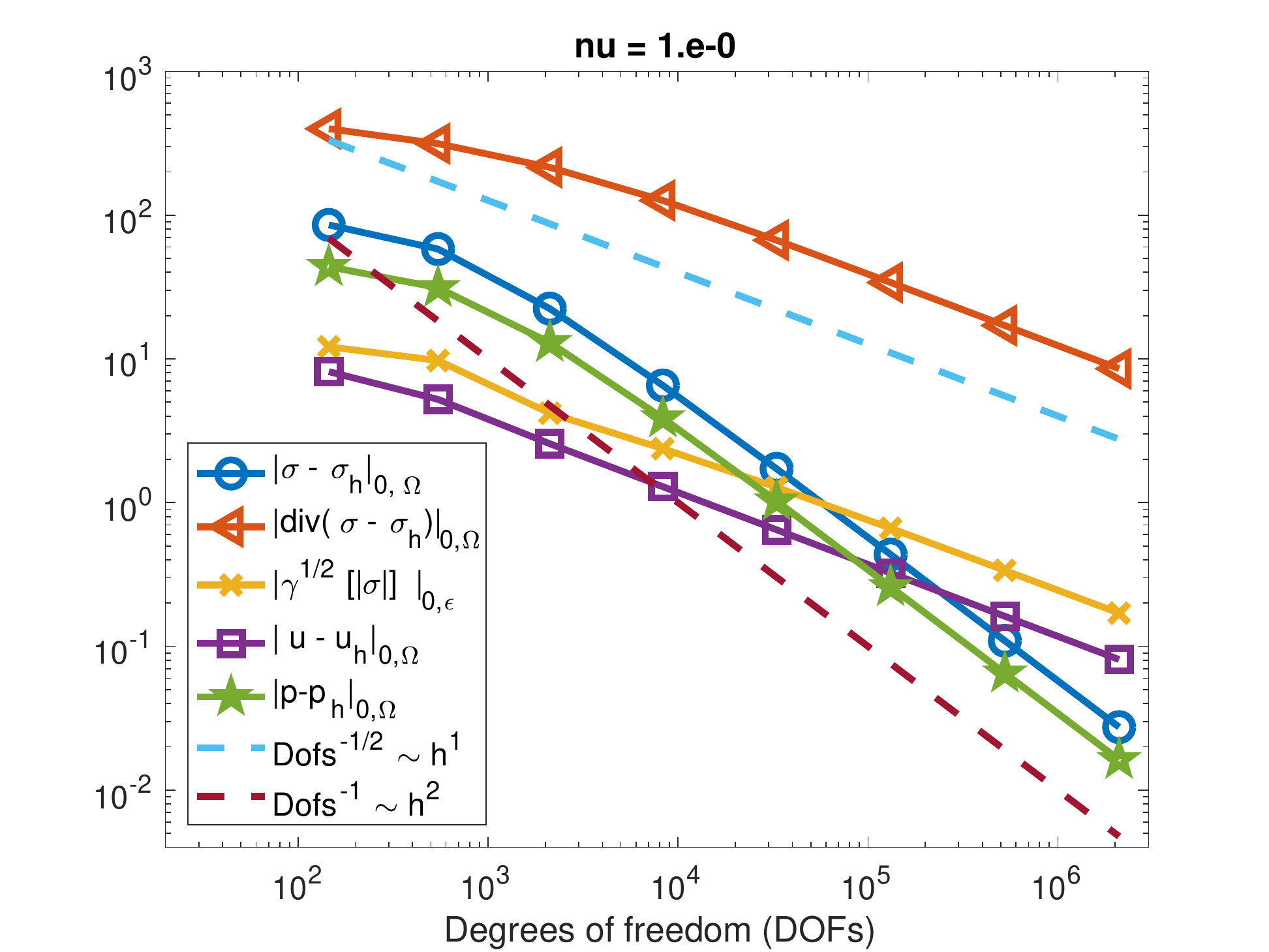}
		\quad
		\includegraphics[scale=0.3]{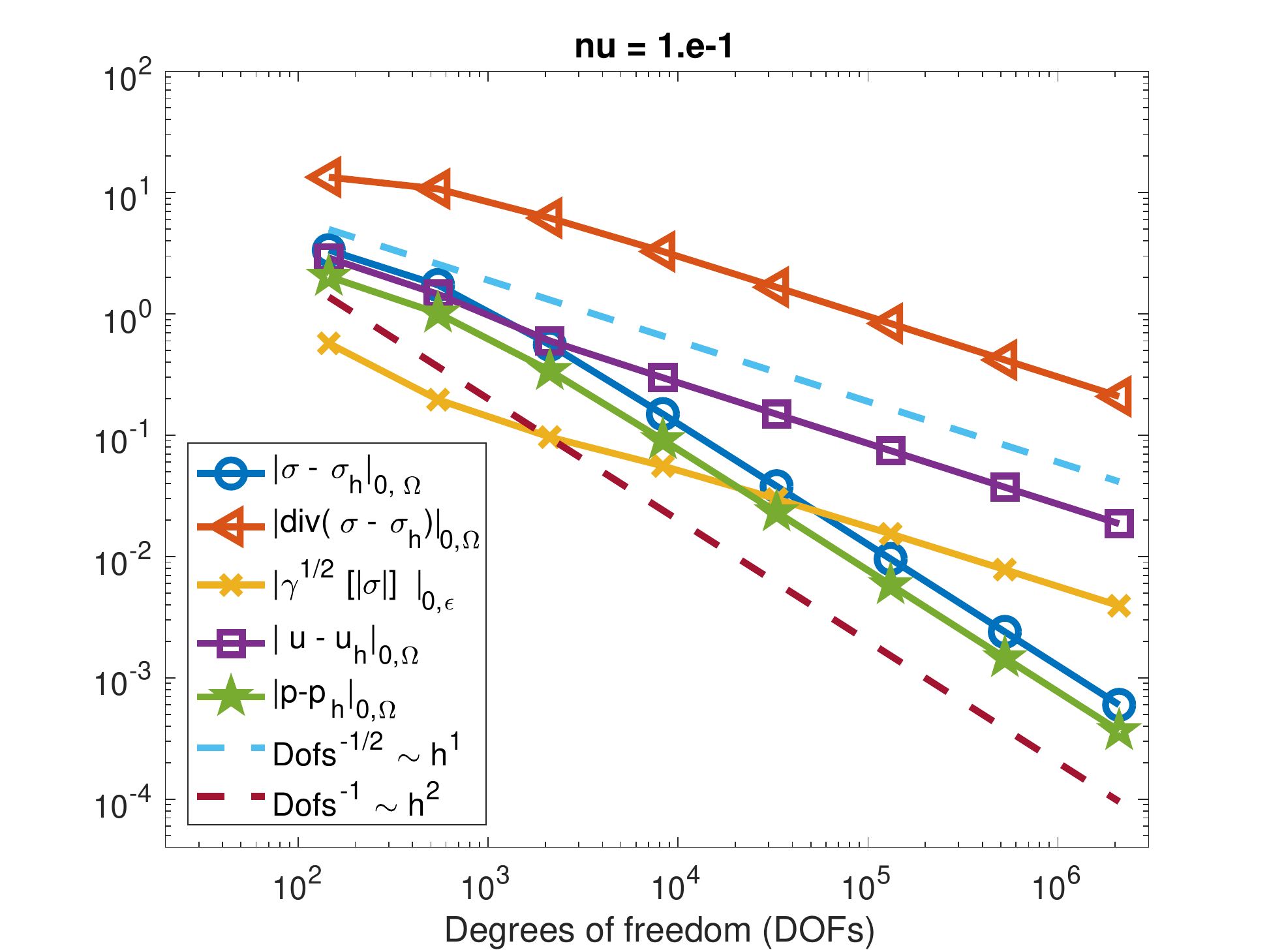}
		
		\includegraphics[scale=0.3]{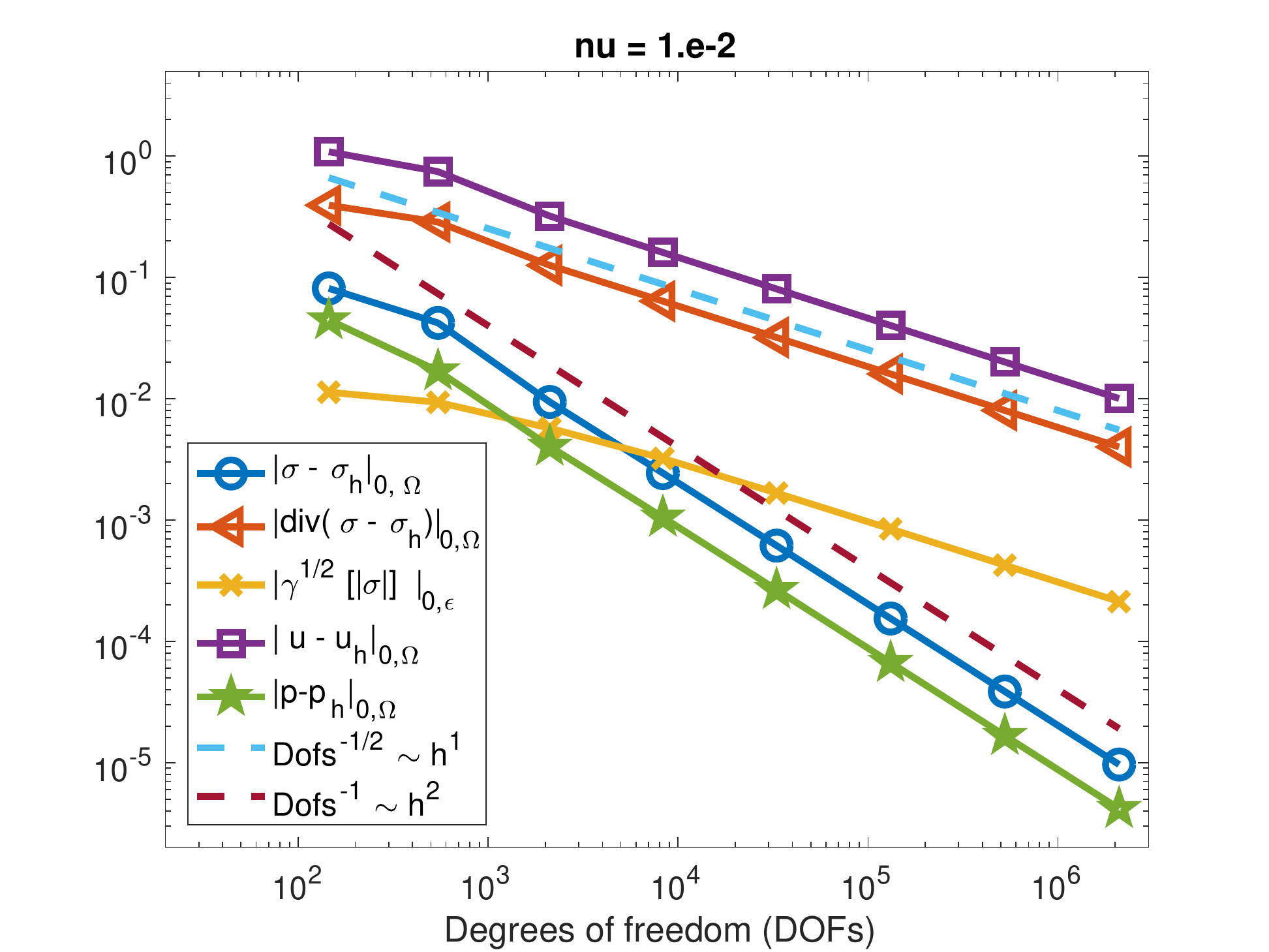}
		\quad
		\includegraphics[scale=0.3]{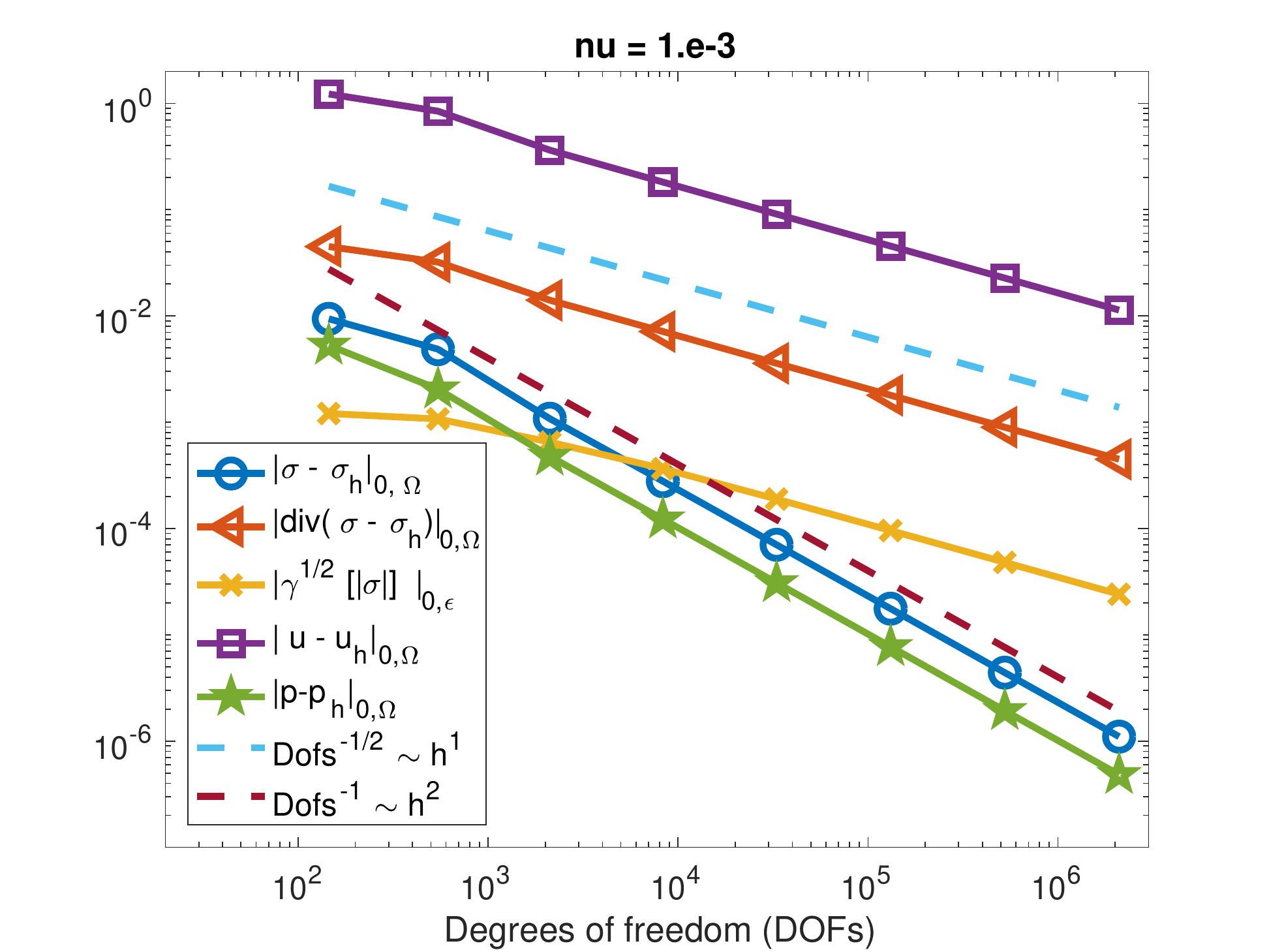}
		
		\includegraphics[scale=0.3]{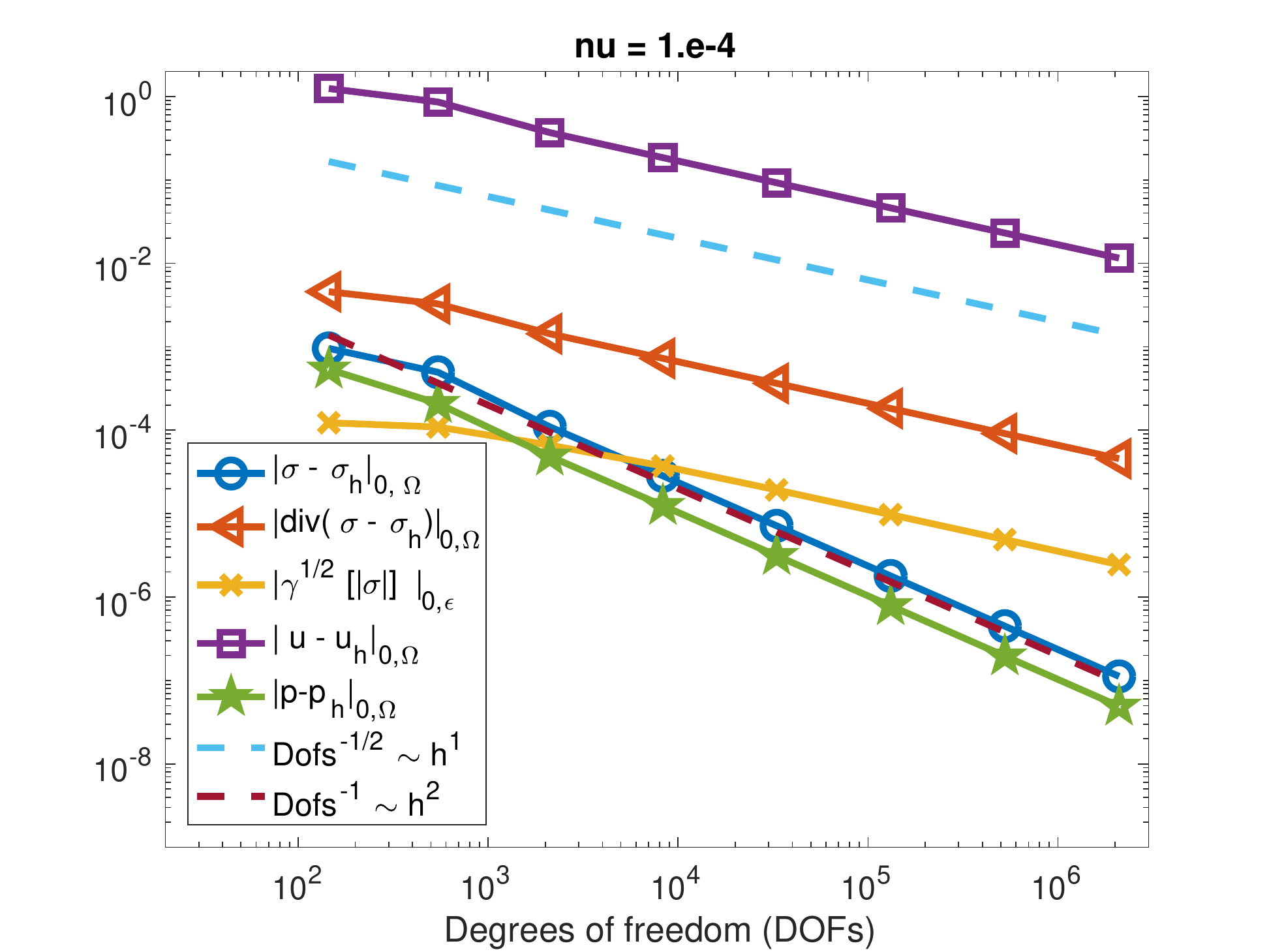}
		\quad
		\includegraphics[scale=0.3]{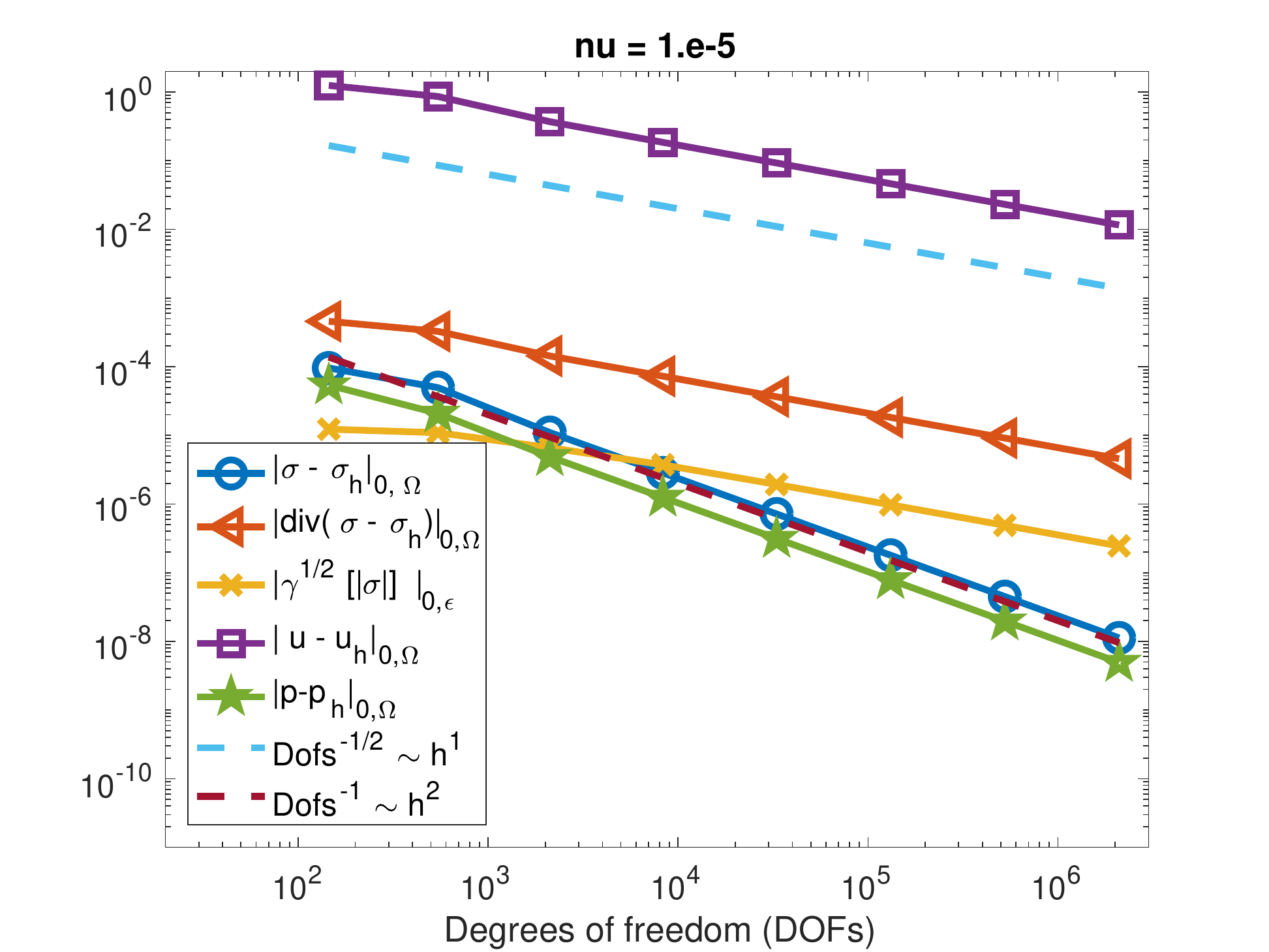}
		\caption{S1. Decay of errors for $\nu \in \{1, 10^{-1}, \ldots, 10^{-5}\}.$}
		\label{figS1_EOC}
\end{figure}

In Figure~\ref{figS1_VP}, we show the velocity stream lines and pressure at iteration 4, 
for the viscosity values analyzed. We do not appreciate any instability featured caused by
a small viscosity. 

\begin{figure}[hbt!]
		\centering
		\includegraphics[scale=0.12]{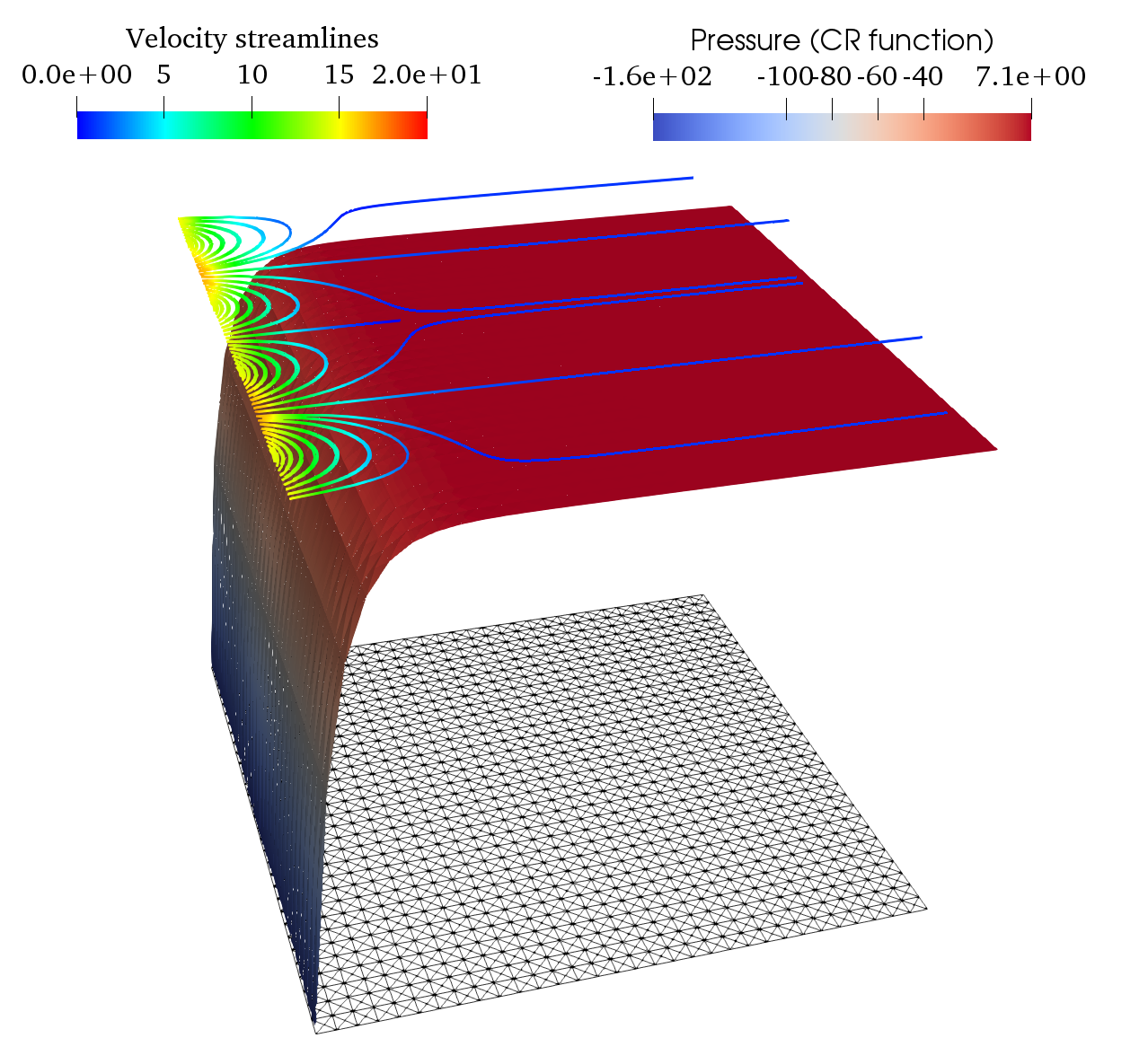}
		\quad
		\includegraphics[scale=0.12]{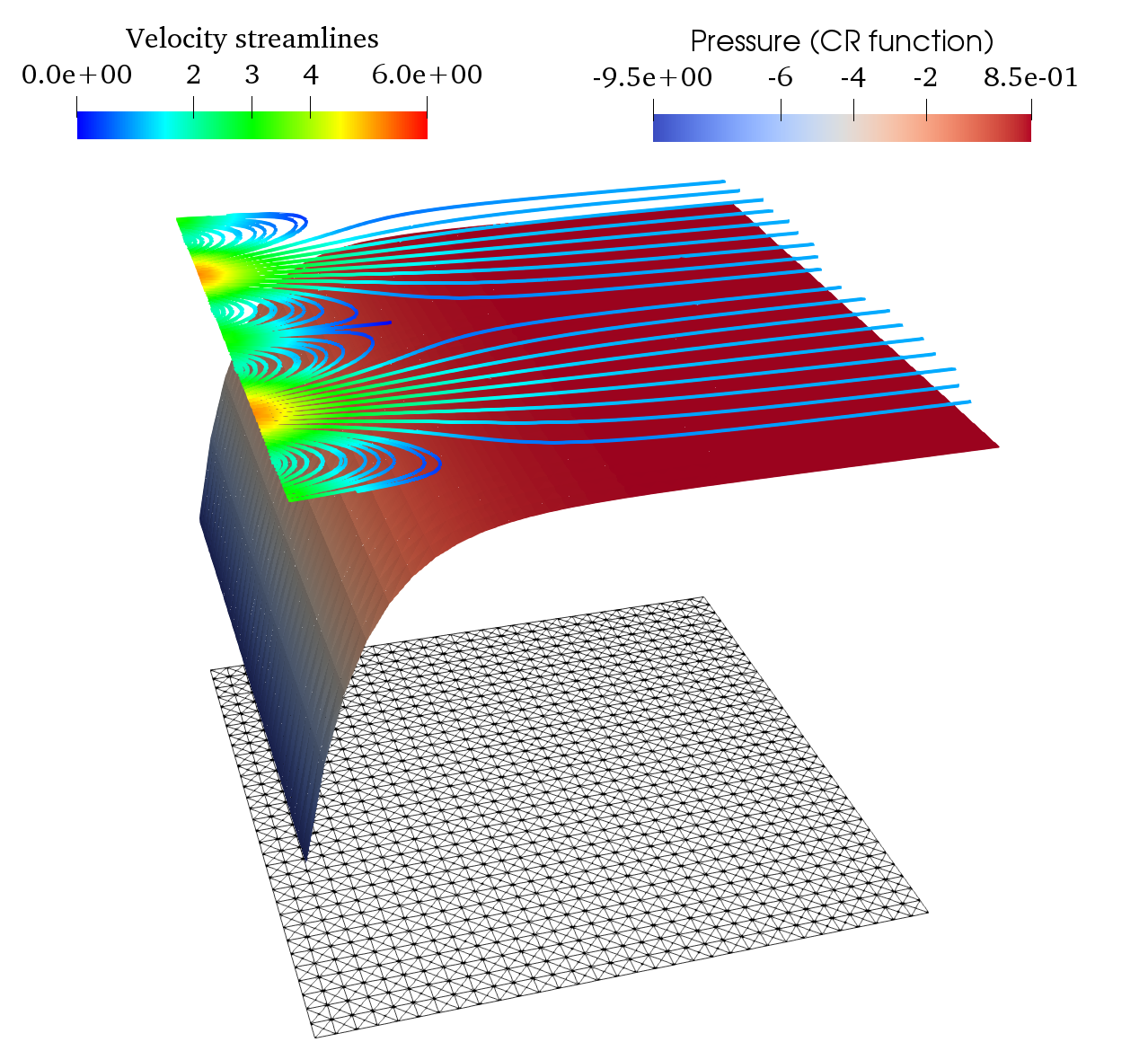}
		
		\includegraphics[scale=0.12]{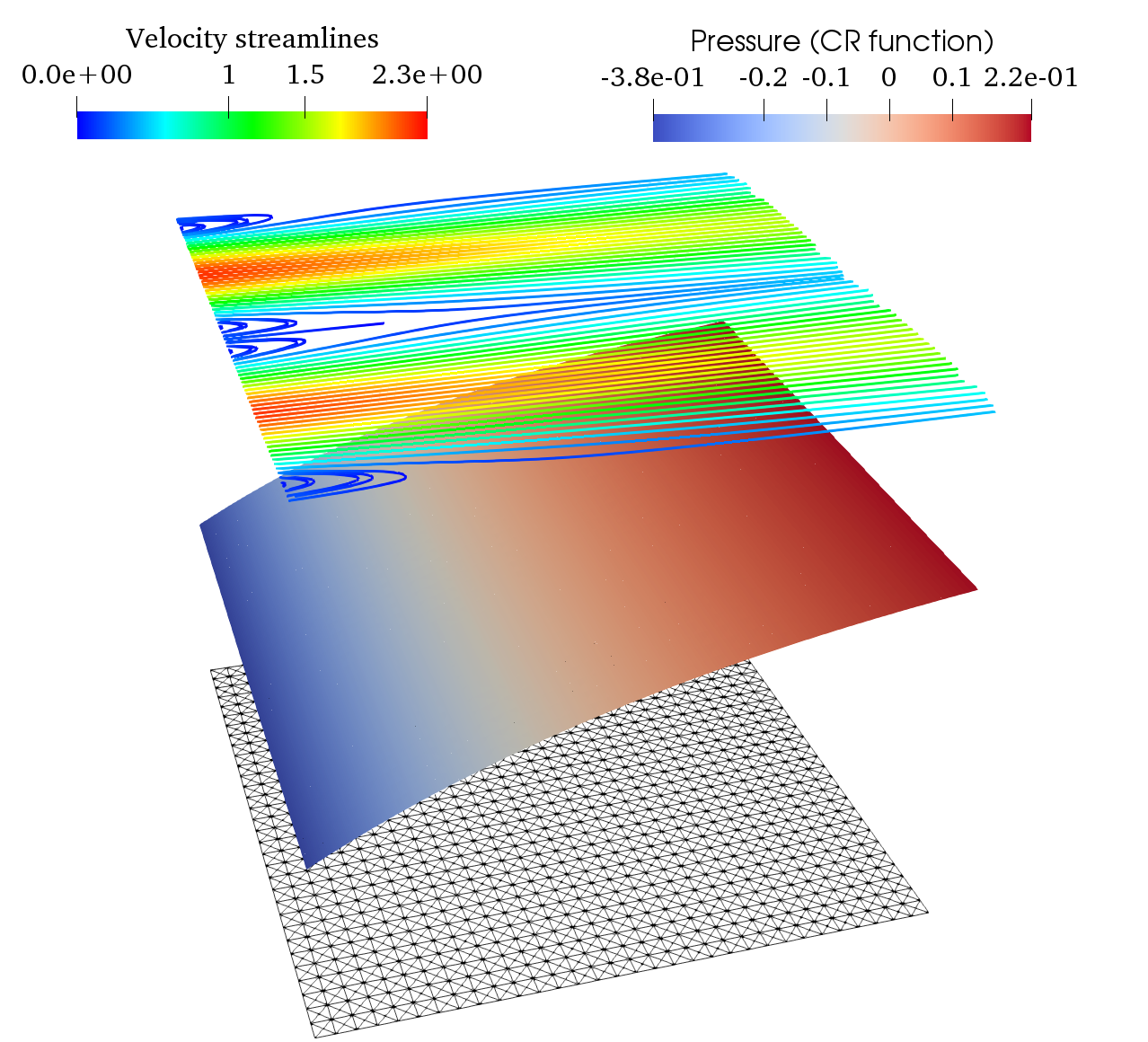}
		\quad
		\includegraphics[scale=0.12]{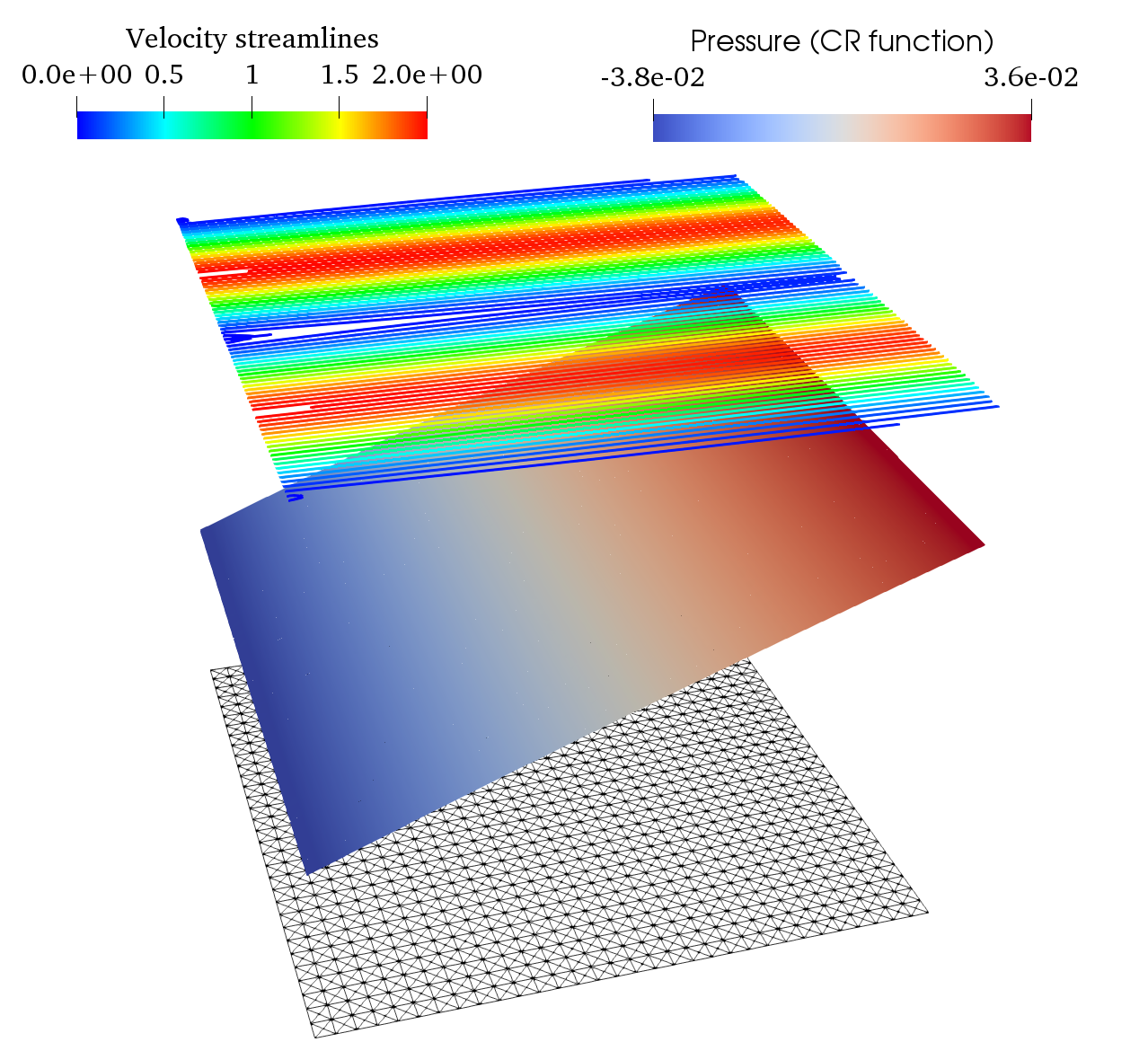}
		
		\includegraphics[scale=0.12]{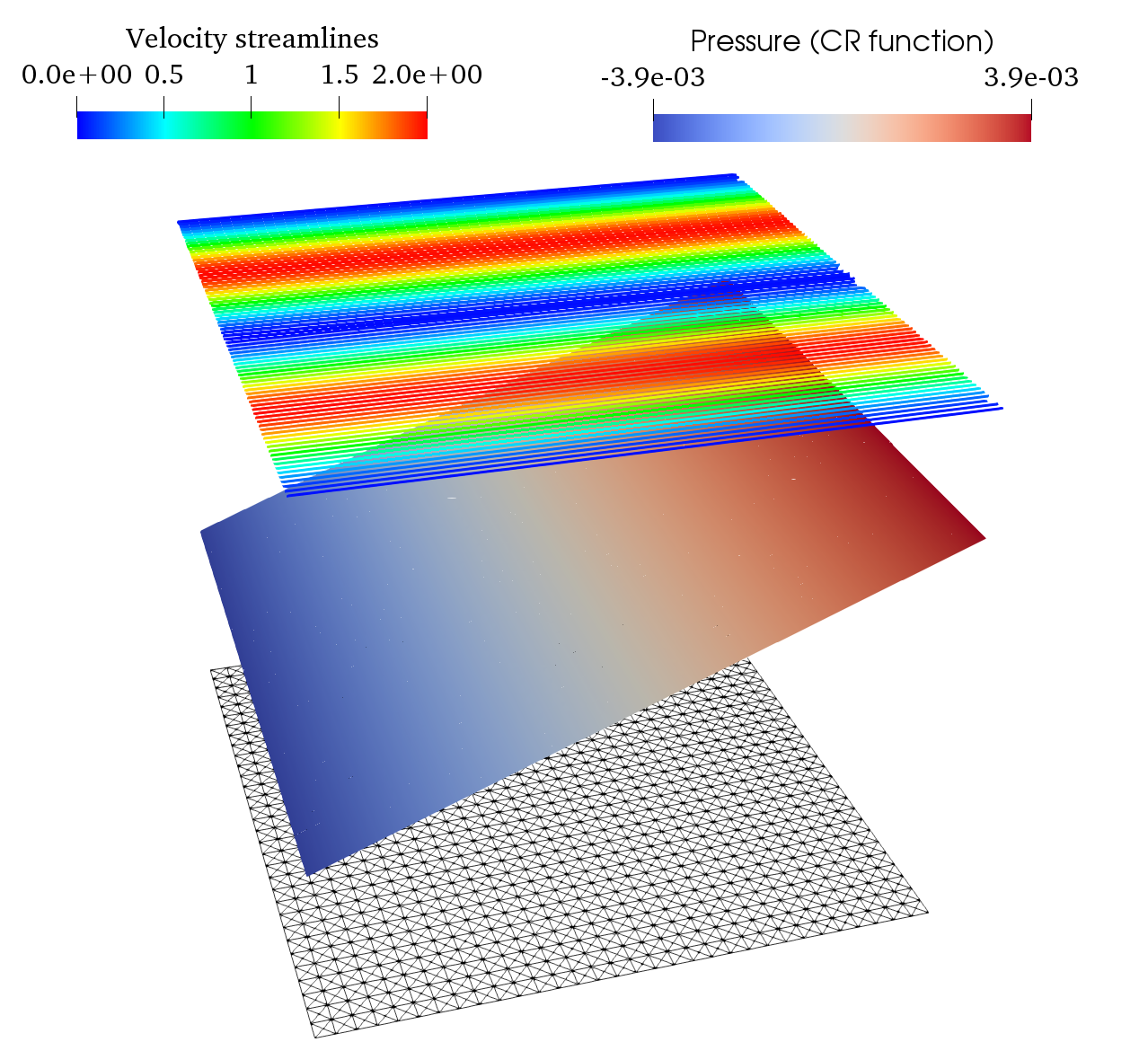}
		\quad
		\includegraphics[scale=0.12]{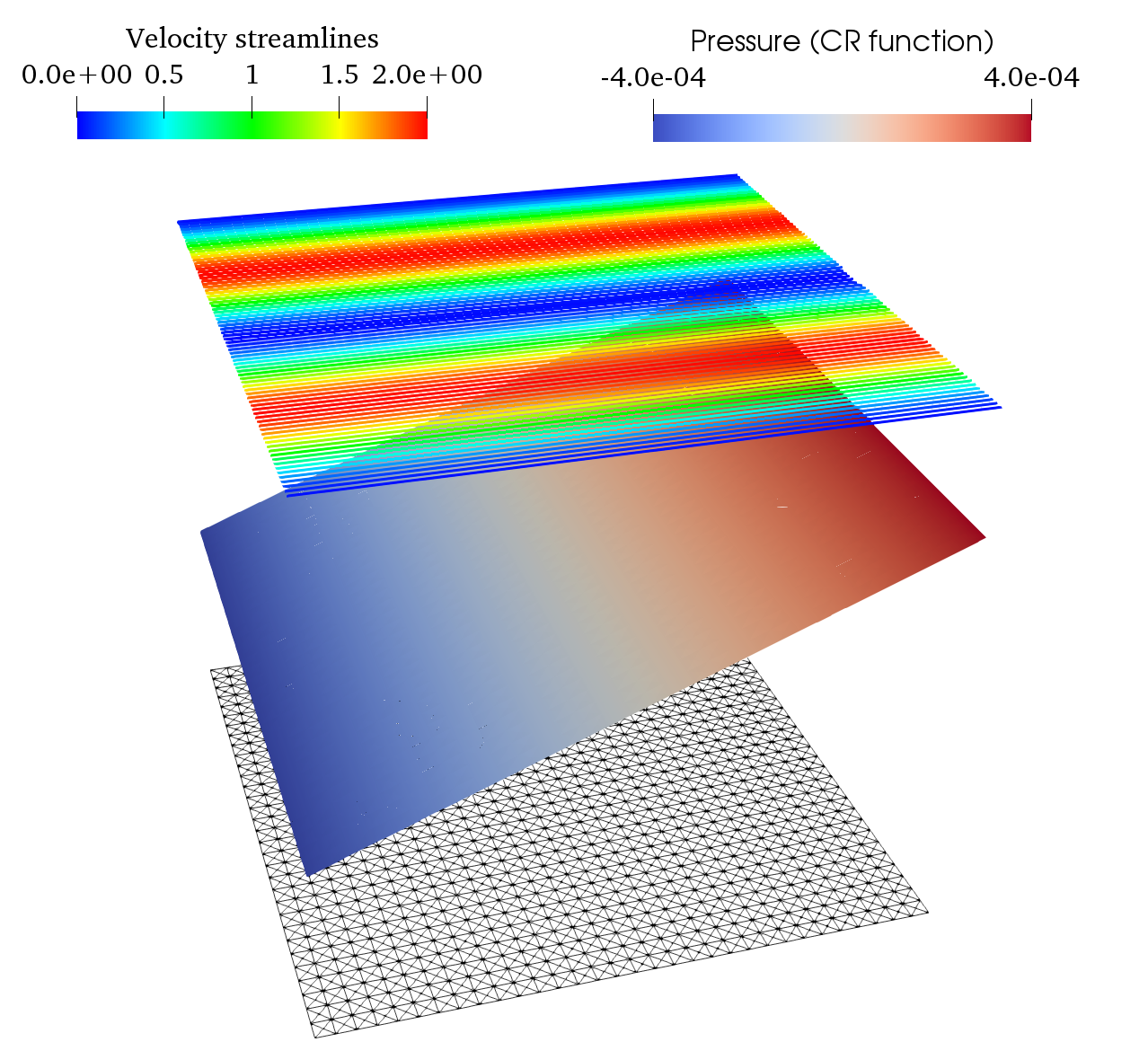}
		\caption{S1. Velocity stream lines and pressure at iteration 4, for 
		$\nu \in \{1, \dots, 10^{-5}\}$ (left to right, top to bottom).
Note that the scale is different.}
		\label{figS1_VP}
\end{figure}

\FloatBarrier
\subsubsection{S2. M-shaped domain}
We now consider the singular solution of \eqref{eqStokes2} proposed in \cite[page 113]{verf1996review}.
Again, let $\Omega$ be M-shaped domain shown in Figure~\ref{figInitialMeshes}(left), and take $\bff \equiv\boldsymbol0$ and $\nu=1.$
Then, if $(r,\theta)$ denote the polar coordinates, we impose an appropriate 
boundary condition for $\bfu$ so that:
\begin{equation}\label{eqS2a}
\begin{split}
		\bfu(r,\theta) &= r^{\lambda}
\begin{bmatrix}
(1 + \lambda) \sin (\theta) \Psi(\theta) + \cos (\theta) \Psi'(\theta) \\
\sin(\theta) \Psi'(\theta) - (1+\lambda) \cos(\theta) \Psi(\theta)
\end{bmatrix},
\\
p(r,\theta) &= -r^{\lambda-1}\left( (1+\lambda)^2 \Psi'(\theta) + \Psi'''(\theta) \right)/(1-\lambda),
\end{split}
\end{equation}
where 
\begin{equation}\label{eqS2b}
\begin{split}
\Psi(\theta) = &\sin((1+\lambda) \theta) \cos(\lambda \omega)/(1+\lambda) - \cos((1+\lambda)\theta)
\\ & - \sin((1-\lambda) \theta ) \cos(\lambda \omega)/(1-\lambda) + \cos(1-\lambda) \theta)\,.
\end{split}
\end{equation}
Here, $\omega = \frac{3\pi}{2}$ and the coefficient $\lambda$ is the smallest positive solution of:
\[
	\sin(\lambda \omega) + \lambda \sin (\omega) = 0 \quad \Rightarrow \quad \lambda\approx 0.54448373\ldots
\]
We emphasize that $(\bfu,p)$ are singular functions, and at the origin 
$\bfu \in H^{1+\lambda}$ and $p, \bsi\in H^{\lambda}.$
Table~\ref{tabS2} gives the individual errors and the corresponding rates.
Again, we omit the divergence error because its size is close  to the rounding unit.

As in the Poisson problem, we note that the jump and the $L_2$-error of $\bfu$ decay with rate $h^1.$
The $L_2$ error of $\bsi$ decay according to the regularity of the exact solution.
Figure~\ref{figS2_VP} show the velocity and pressure at iteration 3 and 6. We observe some oscillations in the first steps that diminish later as the mesh is refined.

\begin{table}[hbt!] \scriptsize
\centering
\caption{Error behavior for S2  problem \eqref{eqS2a}-\eqref{eqS2b} }
\label{tabS2}
\begin{tabular}{rrcccccccc}
\toprule
iter & DOFs 
& $\Lnorm{\bsi-\bsi_h}{\Omega}$ & EOC 
& $\Lnorm{\gamma^{1/2} \jump{\bsi_h}}{\cE_I}$ & EOC
& $\Lnorm{\bfu-\bfu_h}{\Omega}$ & EOC
& $\Lnorm{p-p_h}{\Omega}$ & EOC\\
\midrule
0 &       117 	 &	 3.995\hspace{7mm}  &  -- 	&      1.002\hspace{7mm}  & --  &		  6.341$\cdot10^{-1}$ &  --  	&	  2.612\hspace{7mm}  & --  \\
1    &    425  	 &	 2.781\hspace{7mm}  &  0.56 &	 6.451$\cdot10^{-1}$  &0.68  &		  3.377$\cdot10^{-1}$ &  0.98  	&	  1.769\hspace{7mm} &  0.60 \\
2  &     1617 	 &	 1.851\hspace{7mm} &  0.61 &	 3.515$\cdot10^{-1}$  &0.91  &		  1.725$\cdot10^{-1}$ & 1.01  	&	  1.145\hspace{7mm}  & 0.65 \\
3  &     6305  &    1.232\hspace{7mm}  &  0.60 &	 1.815$\cdot10^{-1}$&0.97  &		  8.677$\cdot10^{-2}$  & 1.01  &		  7.476$\cdot10^{-1}$  & 0.63 \\
4  &    24897 &    8.289$\cdot10^{-1}$ &  0.58 &	 9.194$\cdot10^{-2}$  &0.99  &		  4.341$\cdot10^{-2}$   &1.01  &		  4.967$\cdot10^{-1}$  & 0.60 \\
5  &    98945 &    5.622$\cdot10^{-1}$ &  0.56 &   4.618$\cdot10^{-2}$  &1.00  &		  2.168$\cdot10^{-2}$  & 1.01 & 		  3.346$\cdot10^{-1}$ & 0.57 \\
6 &   394497 &    3.833$\cdot10^{-1}$ &  0.55 &	 2.312$\cdot10^{-2}$  &1.00  &		  1.082$\cdot10^{-2}$ & 1.00 & 		  2.273$\cdot10^{-1}$   &0.56 \\
7 &   1575425&   2.620$\cdot10^{-1}$ &  0.55& 	 1.155$\cdot10^{-2}$  &1.00&  		  5.406$\cdot10^{-3}$  & 1.00 &  	 1.551$\cdot10^{-1}$ &    0.55 \\
\bottomrule
\end{tabular}
\end{table}

\begin{figure}[hbt!]
		\centering
		\includegraphics[scale=0.12]{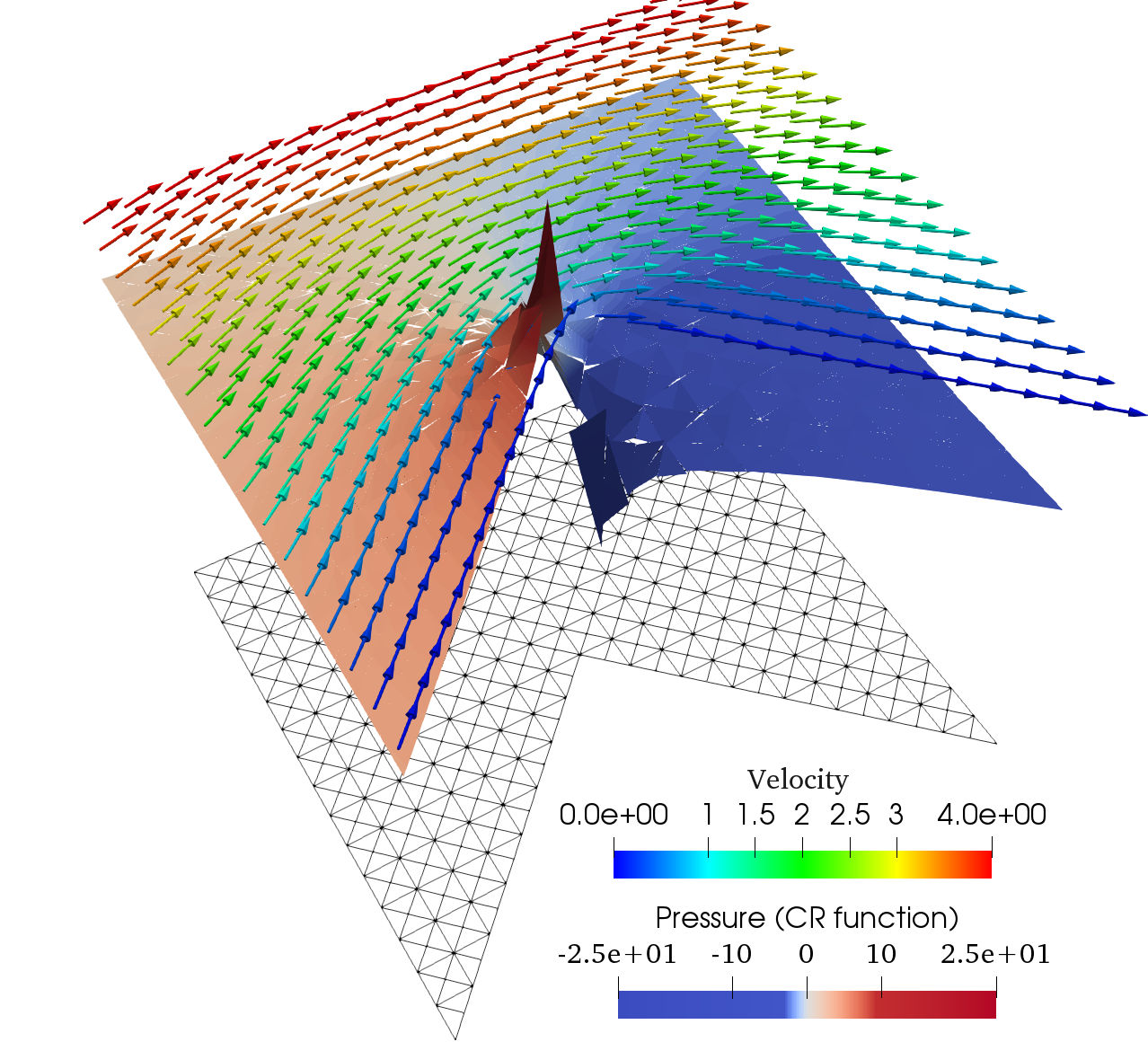}
		\quad
		\includegraphics[scale=0.12]{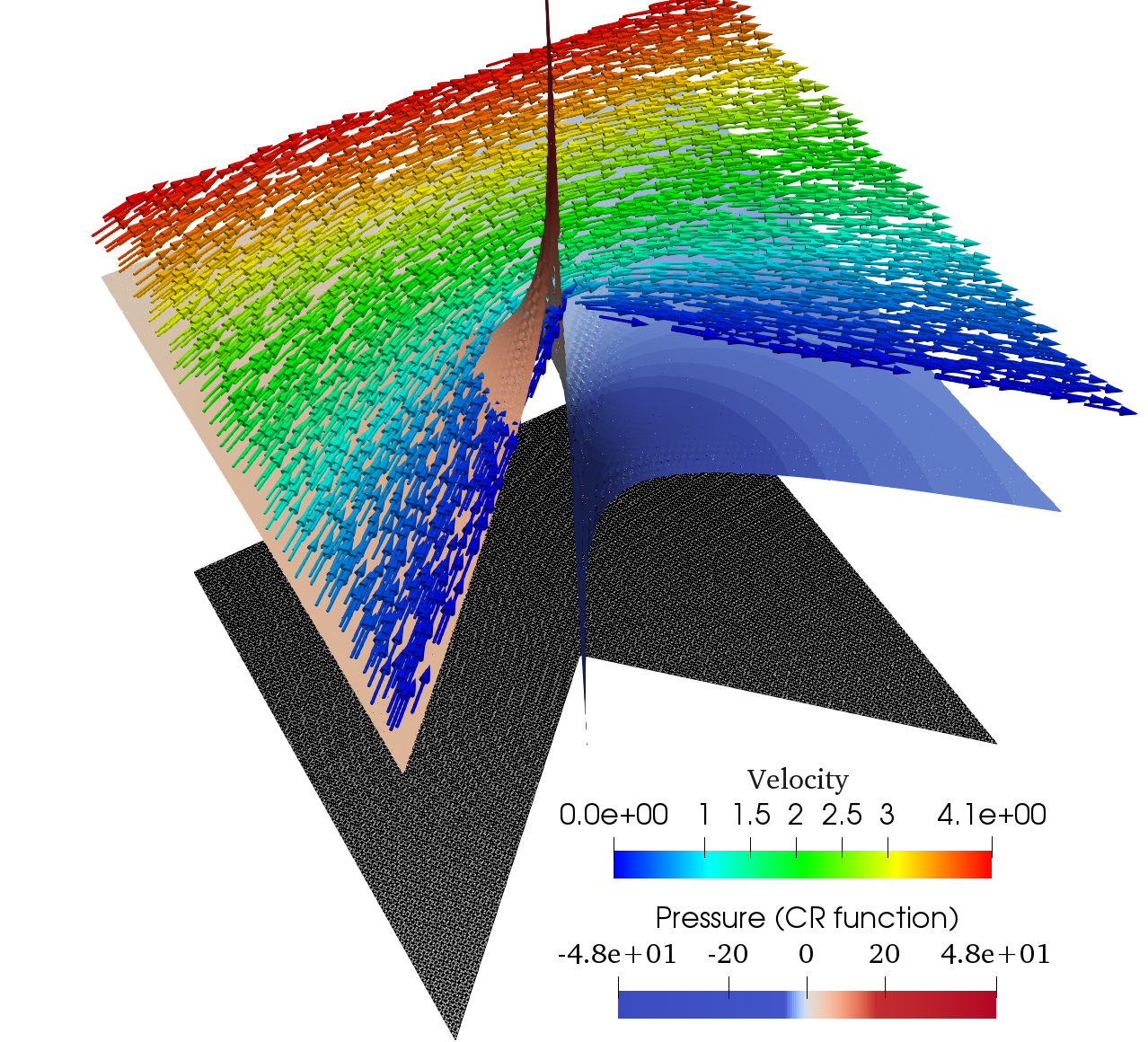}
		
		\caption{S2. Velocity  and pressure at iteration 3 and 6.}
		\label{figS2_VP}
\end{figure}

\subsubsection{S3. Crack domain}
This last example is taken from \cite[page 113]{verf1996review}.
We replicate the previous example on the crack domain, Figure~\ref{figInitialMeshes}(right)
with $(\bfu,p)$ given by \eqref{eqS2a} and 
\begin{equation}\label{eqS3}
\Psi(\theta) = 3 \sin(\theta/2) - \sin (3\theta/2), \quad
\lambda = 1/2, \quad \omega = 2 \pi  \,.
\end{equation}

Table~\ref{tabS3} reports the error decay.
Again, we leave out the divergence error as it is close to the rounding unit.
Although, this example is not covered by our theory, we obtain results equivalent to those in  Section~\ref{S:P3}. 
We remark that the observed  rates for the $L_2$ errors in $\bsi$ and  $p$ are slightly better than the expected value $(0.5).$ 
Finally,  Figure~\ref{figS3_VP}  shows the velocity-field and pressure at iteration 2 and 5.

\begin{table}[hbt!] \scriptsize
\centering
\caption{Error behavior for S3  problem \eqref{eqS2a}-\eqref{eqS3}}
\label{tabS3}
\begin{tabular}{rrcccccccc}
\toprule
iter & DOFs 
& $\Lnorm{\bsi-\bsi_h}{\Omega}$ & EOC 
& $\Lnorm{\gamma^{1/2} \jump{\bsi_h}}{\cE_I}$ & EOC
& $\Lnorm{\bfu-\bfu_h}{\Omega}$ & EOC
& $\Lnorm{p-p_h}{\Omega}$ & EOC\\
\midrule
0&        153 	&	 9.801\hspace{7mm} &  -- 		  &			 1.218\hspace{7mm}  & -- &  		  8.093$\cdot10^{-1}$  & -- &  		    6.657\hspace{7mm}   & --  \\
1 &       561 	&	 8.096\hspace{7mm} &  0.29 		  &		 8.714$\cdot10^{-1}$ & 0.52  &		  4.669$\cdot10^{-1}$   &0.85 & 		    5.449\hspace{7mm}  & 0.31 \\
2   &    2145 	&	 5.685\hspace{7mm} & 0.53 		 &  		 5.362$\cdot10^{-1}$ & 0.72  &		  2.580$\cdot10^{-1}$   &0.88  		  &3.744\hspace{7mm} & 0.56\\ 
3  &     8385 	&	 3.750\hspace{7mm}  & 0.61 		&  	        3.107$\cdot10^{-1}$  &0.80  	      & 1.363$\cdot10^{-1}$    & 0.94  		 & 2.395\hspace{7mm}   &0.66\\ 
4  &    33153 &	 2.464\hspace{7mm} &  0.61 	&	 		 1.772$\cdot10^{-1}$ & 0.82  &		  7.034$\cdot10^{-2}$ &  0.96  	&	    1.520\hspace{7mm}  & 0.66 \\
5  &   131841 &	 1.649\hspace{7mm}  & 0.58 &		 	        1.008$\cdot10^{-1}$  &0.82  		&3.584$\cdot10^{-2}$  & 0.98  		&  9.861$\cdot10^{-1}$   &0.63 \\
6  &   525825 &	 1.128\hspace{7mm}  & 0.55 &		  		 5.766$\cdot10^{-2}$ & 0.81  	&	  1.813$\cdot10^{-2}$ &  0.99 & 		    6.591$\cdot10^{-1}$ &  0.58 \\
7  &  2100225& 	 7.844$\cdot10^{-1}$  & 0.53& 	3.320$\cdot10^{-2}$ & 0.80  		&9.133$\cdot10^{-3}$  & 0.99  	&	    4.514$\cdot10^{-1}$  & 0.55 \\
\bottomrule
\end{tabular}
\end{table}

\begin{figure}[hbt!]
		\centering
		\includegraphics[scale=0.12]{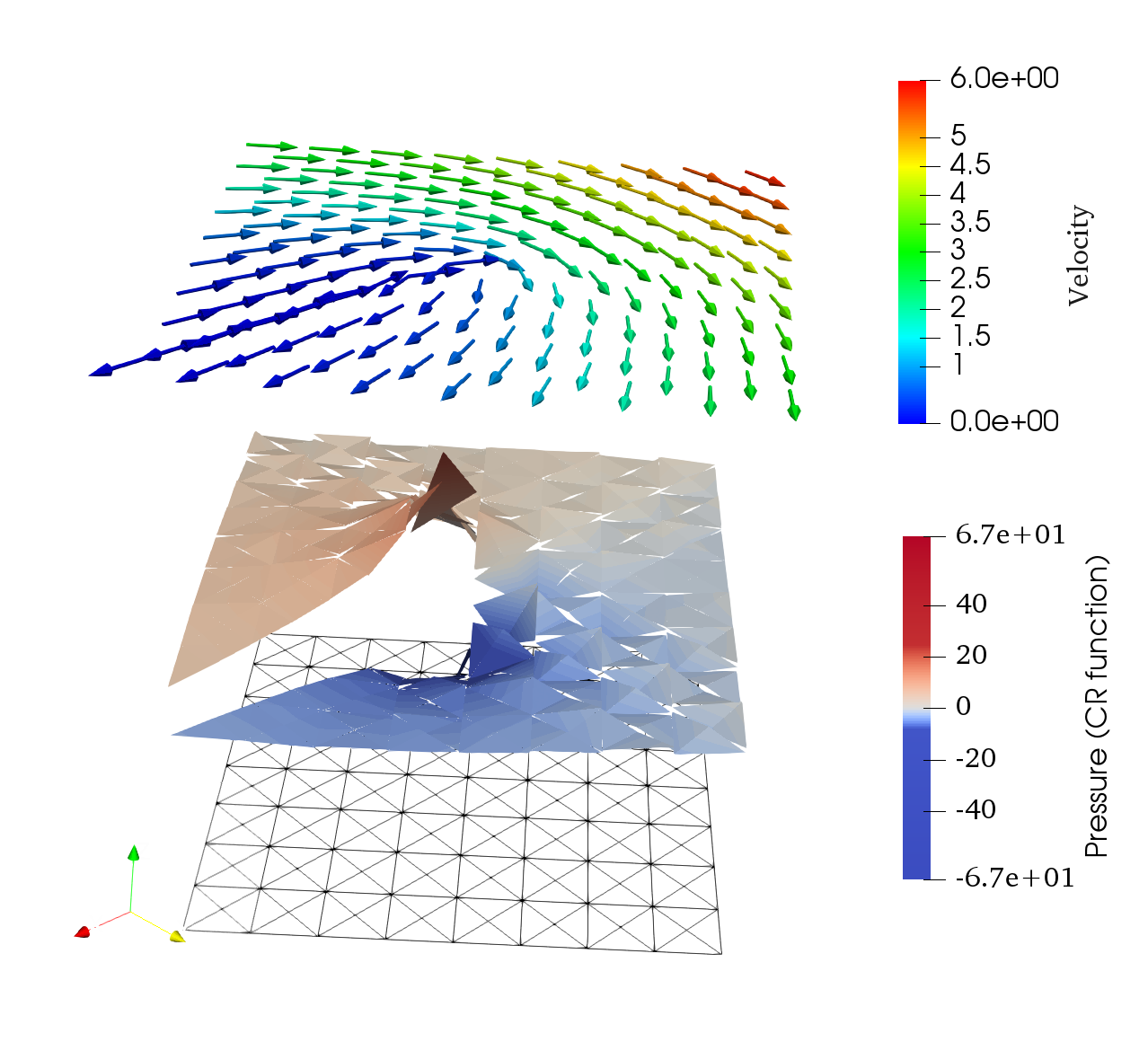}
		\quad
		\includegraphics[scale=0.12]{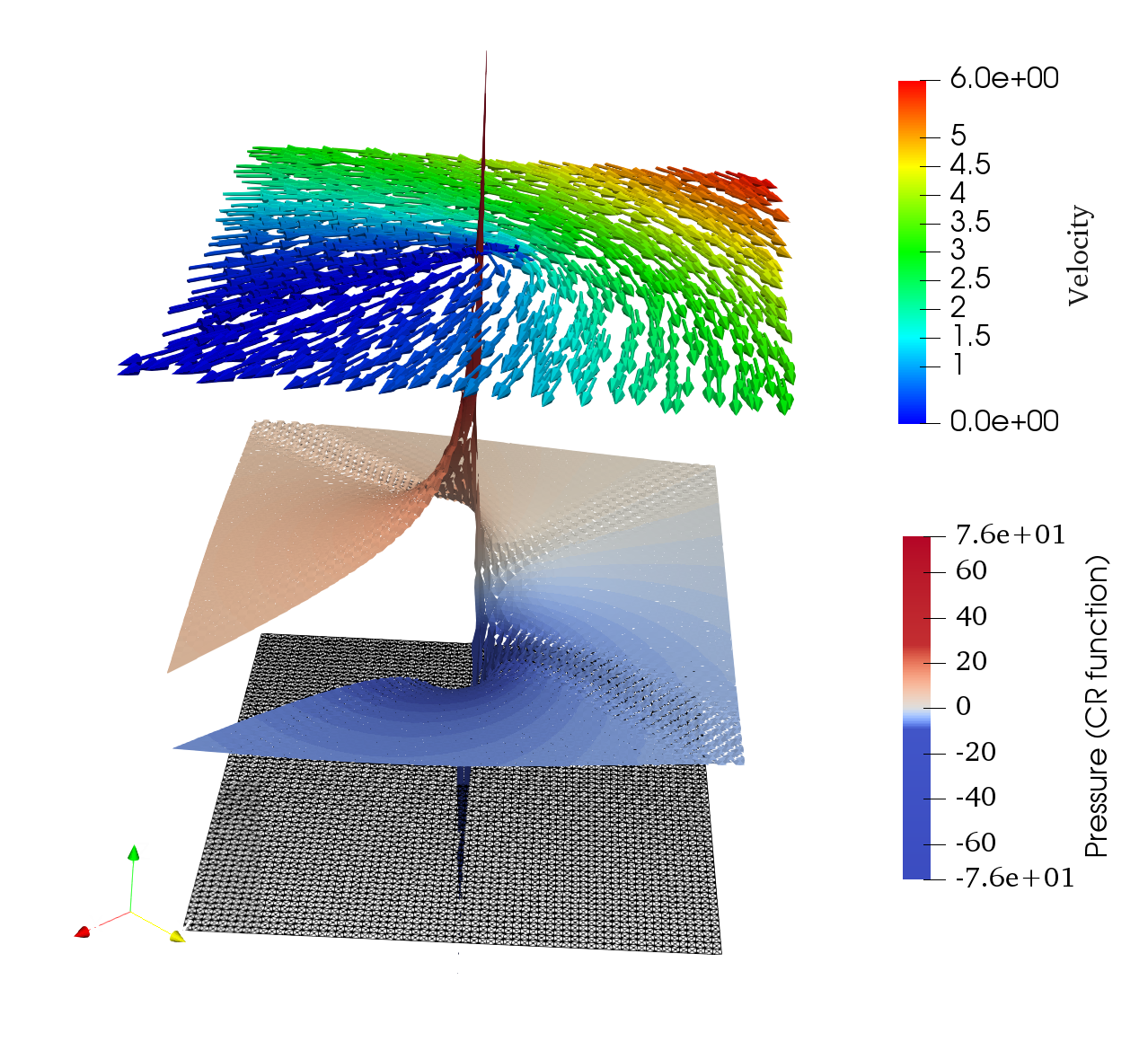}
		
		\caption{S3. Velocity  and pressure at iteration 2 and 5.}
		\label{figS3_VP}
\end{figure}

\section{Final comments and conclusions}\label{section-5}
  We developed an \emph{a priori} error  analysis for a low-order non-conforming method applied to the dual-mixed variational formulation of the Stokes and Poisson  problem, and thus also for Darcy flow.
 Our \emph{a priori} error estimates were  established using Crouzeix-Raviart elements to approximate the $\bsi=\nabla u$  and piecewise constants to approximate $u.$
We proved \emph{a priori} estimates for smooth solutions.
Low regularity and \emph{a posteriori} estimates are part of a forthcoming paper.
However, the divergence error was proven to converge optimally for any regularity.
  The experimental approximation errors of our non-conforming scheme
   improve those of the conforming Raviart-Thomas approach.

Finally, an advantage of this approach over the one with Raviart-Thomas elements,
 is that, the CR-element is defined component-wise and not vector-wise.
 In fact, this leaves more flexibility for a symmetry restriction when dealing with elasticity equations, which is much more difficult to 
 achieve with the Raviart-Thomas element.
 In particular for the lowest order  Raviart-Thomas pair, this could only be done  for a tensor with constant entries. 
 We therefore, intent to extend the CR-element to the  elasticity problem  in a separate paper.

%
 \subsection*{Acknowledgement}
 We are grateful to Erik Burman for providing us with the reference \cite{BH2005}.





\end{document}